\documentclass[a4paper,11pt]{article}

\usepackage{mathrsfs,xcolor,bm}
\usepackage{epsfig, graphicx}
\usepackage{latexsym,amsfonts,amsbsy,amssymb,ulem}
\usepackage{amsmath,amsthm}

\makeatletter

\@addtoreset{equation}{section}

\renewcommand{\div}[0]{\textnormal{div}\,}
\newcommand{\curl}[0]{\textnormal{\textbf{curl}}\,}

\newcommand{\Hcurl}[1]{\mathbf{H}(\textnormal{\textbf{curl}},#1)}
\newcommand{\Hcurltr}[1]{\mathbf{H}_0(\textnormal{\textbf{curl}},#1)}

\allowdisplaybreaks

\newtheorem{theorem}{Theorem}[section]
\newtheorem{lemma}{Lemma}[section]
\newtheorem{corollary}{Corollary}[section]
\newtheorem{remark}{Remark}[section]

\date {}
\begin{document}




\title{A finite  element method for Maxwell's transmission eigenvalue problem in anisotropic media\thanks{This work was partially supported by Project supported by
the National Natural Science Foundation of China(Grant No.  12361084, 12001130).}}
\author{Jiayu Han \thanks{School of Mathematical Sciences, Guizhou Normal University, 550001, China.({\tt hanjiayu@gznu.edu.cn})} }
\maketitle

{\bf Abstract.}
In this paper, we introduce a  finite element method employing the Ned\'el\'ec element space for solving the Maxwell's  transmission eigenvalue problem in anisotropic media. The well-posedness of the source problems are derived using   $\mathbb T$-coercivity approach.
We discuss the discrete compactness property of the finite element space under the case of  anisotropic coefficients   and conduct  a finite element error analysis for the proposed approach.
Additionally, we present some numerical examples   to  support  the theoretical result.
%
%
%
%
%



\section{Introduction}
The transmission eigenvalue problem is a crucial aspect of inverse scattering theory, serving to estimate the properties of scattering materials  \cite{CakoniCayorenColton,CakoniGintidesHaddar,Sun_1}.
In recent years, significant progress has been achieved  in the mathematical analysis of the transmission eigenvalue problem and its modified versions, along with their practical applications,  as evidenced by recent works  \cite{BonnetChesnelHaddar,CakoniCayorenColton,CakoniHaddar,CakoniGintidesHaddar,ChesnelCiarlet,ColtonKress,Kirsch,PaivarintaSylvester}.
There has been growing interest among researchers in scientific and engineering computing regarding numerical methods for addressing the transmission eigenvalue problem \cite{
ColtonMonkSun,ColtonPaivarintaSykvester,Kleefeld2018IPB,mora,SunZhou2016,xi,xie,Yang2016}.

The Maxwell's transmission eigenvalue problem (MTEP) involves two vectorial functionals in $\Hcurl{\Omega}$ which share a common tangential trace on the boundary.
Curl-conforming methods are direct and robust for  solving   MTEP.  Interested readers can find details on the construction of different curl-conforming schemes in the literatures   \cite{monk2,sun2}. To address the coupled boundary condition of   MTEP,   boundary element methods in presented in \cite{cossonniere,Kleefeld2022nua} have provided  significant and interesting numerical examples.
  Huang et al. \cite{huang} established a finite difference scheme and provided an efficient linear algebraic eigensolver for   MTEP.

Finite element error analysis for   MTEP presents a challenging task due to its non-elliptic and non-self-adjoint nature. The associated numerical methods have limited theoretical results. Enforcing the div-free constraint in the solution space complicates the numerical analysis process.  To avoid this issue,
An et al. \cite{an} transformed the eigenvalue problem of     special coefficients   on spherical domains in three dimensions into a series of one-dimensional eigenvalue problems.
As is shown in \cite{monk2}, 
for a special isotropic  media (i.e., $\bm A=\bm I$ in \eqref{Problem}),    MTEP  can be equivalently transformed as a fourth order curl eigenvalue problem.
   In this case one can derive the regularity of  the solution  space and the compactness property of the solution operator. Consequently, certain numerical treatments, such as the $\mathbf H(\curl^2)$-conforming method \cite{han} and the discontinuous Garlerkin method \cite{han1},  with a rigorous convergence proof for variable coefficient and general polyhedral domains, can be further developed.

However, the existing
numerical results do not cover  the case of anisotropic media, i.e., $\bm A \ne I$ in \eqref{Problem}, for which   MTEP cannot be transformed as a fourth order curl eigenvalue problem. To fill this gap, in this paper we propose a curl-conforming finite element
approximation of  MTEP in an anisotropic media. Firstly, we reformulate  MTEP as a coupled second-order curl eigenvalue problem. Secondly, we enforce the div-free constraint in the weak formulation to ensure the well-posedness of the associated source problem.
For the transmission eigenvalue problem, the  $\mathbb T$-coercivity  theory serves  as a powerful tool for analyzing the existence of eigenvalues (see, e.g., \cite{BonnetChesnelHaddar,BonnetChesnelCiarlet,ChesnelCiarlet}), especially  the existence of Maxwell's transmission eigenvalues (see \cite{chesnel}).
Here we introduce two sesquilinear forms satisfying  the $\mathbb T$-coercivity property in the error analysis.  The solution space involved exhibits proper smoothness properties to guarantee its compact embedding into $\mathbf L^2(\Omega)^2$.  Thirdly, we establish the discrete compactness of the finite element space, which  leads to the collectively compact convergence of the discrete solution operator.   With the aid of  the spectral approximation theory of  Babu\v ska and Osborn \cite{babuska,osborn}, we ultimately prove the optimal convergence order of the finite element eigenpairs. We present numerical examples in three dimensions, encompassing constant coefficient, variable coefficient, and anisotropic coefficient cases, to validate the theoretical result.



To simplify notation, we use the symbols $\mathfrak a\lesssim\mathfrak  b$ and  $\mathfrak a \gtrsim\mathfrak  b$ to denote  that $\mathfrak a\leq C\mathfrak  b$ and $\mathfrak a \geq C \mathfrak b$, respectively,
 where $C$ denotes a generic
positive constant independent of mesh sizes which may vary at  each occurrence.


\section{Maxwell's transmission eigenvalue problem}\label{Transmission_Eigenvalue_Problem}

Consider the Maxwell's transmission eigenvalue problem of finding $k\in \mathbb{C}\backslash\{0\}$,
 $(\bm \omega,\bm \upsilon)\in \Hcurl{\Omega}\times \Hcurl{\Omega}$ such that (c.f. \cite{CakoniGintidesHaddar,chesnel,monk2,sun2})
\begin{equation}\label{Problem}
\left\{
\begin{array}{rcll}
{\curl}\big(\bm A\curl\bm \omega\big)&=&k^2\bm N\bm \omega,&{\rm in}\ \Omega,\\
\curl^2\bm \upsilon &=&k^2\bm \upsilon, &{\rm in}\ \Omega,\\
(\bm \omega-\bm \upsilon)\times\bm\nu&=&0,&{\rm on}\ \partial \Omega,\\
\bm A\curl\bm \omega\times\bm\nu-\curl\bm \upsilon\times\bm\nu&=&0,&{\rm on}\ \partial \Omega,
\end{array}
\right.
\end{equation}
where $\Omega\subset \mathbb R^3$ is a  bounded
simply connected anisotropic medium and $\bm\nu$ is the unit outward normal to the boundary $\partial \Omega$.
The parameters  $\bm A$ and  $\bm N$ reflect  permittivity and  permeability of the medium \cite{chesnel}, respectively.
 Throughout this paper we assume that  the Hermite matrices $\bm A(\bm x)$ and  $\bm N(\bm x)$ satisfy
\begin{equation}\label{cond-A-n}
A_* |\bm\xi|^2<\overline{\bm\xi}\cdot\bm A\bm\xi< A^* |\bm\xi|^2,\quad  N_* |\bm\xi|^2<\overline{\bm\xi}\cdot\bm N\bm\xi< N^* |\bm\xi|^2,\ \ \ \forall \bm\xi\in \mathbb C^3~\text{a.e. in}~\Omega, 
\end{equation}
where   $A_*,A^*,N_*,N^*$ are   positive constants satisfying one of the following cases:


1) $A_*>1$ and $N_*>1$,~~2) $A_*>1$ and $N^*<1$,~~
3) $A^*<1$ and $N^*<1$,~~4) $A^*<1$ and $N_*>1$.
We introduce the following three  sesquilinear forms
\begin{align}
a((\bm \omega,\bm \upsilon),(\bm \omega',\bm \upsilon'))&:=\big(\bm A\curl \bm \omega,\curl \bm \omega'\big)
-\big(\curl \bm \upsilon,\curl \bm \upsilon'\big),\\
c((\bm \omega,\bm \upsilon),(\bm \omega',\bm \upsilon'))&:=\big(\bm N\bm \omega,\bm \omega'\big)-\big(\bm \upsilon,\bm \upsilon'\big),\\
b(( p', q'),(\bm \omega',\bm \upsilon'))&:=c((\nabla p',\nabla q'),(\bm \omega',\bm \upsilon'))=\big(\bm N\nabla p',\bm \omega'\big)-\big(\nabla q',\bm \upsilon'\big).
\end{align}
 Let $\lambda=k^2$ then the  variational form of \eqref{Problem} is to find  $\lambda\in \mathbb C\backslash\{0\}$ and nozero $(\bm \omega,\bm \upsilon) \in \bm{\mathcal{H}}$ such that
\begin{align}\label{Problem_Weak0}
a((\bm \omega,\bm \upsilon),(\bm \omega',\bm \upsilon'))=\lambda c((\bm \omega,\bm \upsilon),(\bm \omega',\bm \upsilon')), \ \ \ \forall (\bm \omega',\bm \upsilon') \in \bm{\mathcal{H}}.
\end{align}
We can see that $a(\cdot,\cdot)$ and $c(\cdot,\cdot)$ satisfy the symmetry properties
\begin{align}\label{sp}
a((\bm \omega,\bm \upsilon),(\bm \omega',\bm \upsilon'))=\overline{a((\bm \omega',\bm \upsilon'),(\bm \omega,\bm \upsilon))},\quad c((\bm \omega,\bm \upsilon),(\bm \omega',\bm \upsilon'))=\overline{c((\bm \omega',\bm \upsilon'),(\bm \omega,\bm \upsilon))}.
\end{align}
Define
\begin{align*}
&\bm{\mathcal{H}} := \{(\bm{w}',\bm{v}')\in \Hcurl{\Omega}\times\Hcurl{\Omega} : \bm{w}'-\bm{v}'\in\Hcurltr{\Omega}\},\\
&\mathcal  Q:=\{(p',q')\in H^1(\Omega)\times H^1(\Omega):\int_{\partial\Omega} p'=\int_{\partial\Omega} q'=0 \text{ and } p'-q'\in H^1_0(\Omega)\},\\
&\bm{\mathcal{H}}_{0}:=\{(\bm \omega',\bm \upsilon')\in \bm{\mathcal{H}}: b(( p', q'),(\bm \omega',\bm \upsilon'))=0,\quad\forall (p', q')\in \mathcal Q\},
\end{align*}
equipped with the norms
\begin{align*}
&\|(\bm \omega',\bm \upsilon')\|_{\bm{\mathcal{H}}}:=\|\bm \omega'\|_{\curl}+\|\bm \upsilon'\|_{\curl}^2,\quad
\|(p',q')\|_{\mathcal Q}:=\|p'\|_{1,\Omega}+\|q'\|_{1,\Omega},
\end{align*}
${\rm and}~\|(\bm \omega',\bm \upsilon')\|_{\bm{\mathcal{H}}_{0}}:=\|(\bm \omega',\bm \upsilon')\|_{\bm{\mathcal{H}}},$
respectively. Let  $\|\cdot\|$ be the norm in $\mathbf L^2(\Omega)$ and let $\|(\bm \omega',\bm \upsilon')\|:=\|\bm \omega'\|+\|\bm \upsilon'\|$ be  the  norm in $\mathbf L^2(\Omega)\times \mathbf L^2(\Omega)$. We also use the symbol $\|\cdot\|_{s,\Omega}$ $(s>0)$ to denote the norm in the standard Sobolev space $\mathbf H^s(\Omega)$ and further $\mathbf H^s(\curl,\Omega):=\{\bm v\in \mathbf H^s(\Omega): \curl \bm v\in \mathbf H^s(\Omega)\}$.

In order to analyze the well-posedness of  the  problem \eqref{Problem_Weak0}, we introduce
the so-called $\mathbb T$-coercivity for the sesquilinear forms $a$ and $c$, which was proposed by
the literatures  \cite{BonnetChesnelHaddar,BonnetChesnelCiarlet,ChesnelCiarlet}. 
In this paper, the notations $\mathbb T_a$  and $\mathbb T_c$ are isomorphic operators from $\bm{\mathcal{H}}$ to $\bm{\mathcal{H}}$, which are
defined as
\begin{eqnarray}
\mathbb T_a(\bm \omega',\bm \upsilon')=\left\{
\begin{array}{cc}
(\bm \omega', 2\bm \omega'-\bm \upsilon'),&\text{if $A_*>1$},\\
(\bm \omega'-2\bm \upsilon',-\bm \upsilon'),&\text{if $A^*<1$},
\end{array}
\right.
\mathbb T_c(\bm \omega',\bm \upsilon')=\left\{
\begin{array}{cc}
(\bm \omega', 2\bm \omega'-\bm \upsilon'),&\text{if $N_*>1$},\\
(\bm \omega'-2\bm \upsilon',-\bm \upsilon'),&\text{if $N^*<1$}.
\end{array}
\right.
\end{eqnarray}
For  isomorphic operators of more general $\mathbb T$-coercivity properties, we refer the readers to the work of \cite{chesnel}.
  Similarly as  \cite{BonnetChesnelHaddar,BonnetChesnelCiarlet,ChesnelCiarlet}
we  can prove  the following lemma.
\begin{lemma}\label{lemma2.1}The sesquilinear forms $a$ and $c$ satisfy the following
$\mathbb T$-coercivity properties
\begin{eqnarray}\label{tc1}
|a((\bm \omega',\bm \upsilon'),\mathbb T_a(\bm \omega',\bm \upsilon'))|
\geq C(\|\curl\bm \omega'\|_{}^2+ \|\curl\bm \upsilon'\|_{}^2),\quad \forall (\bm \omega',\bm \upsilon')\in\bm{\mathcal{H}}
\end{eqnarray}
and
\begin{eqnarray}\label{tc2}
|c((\bm \omega',\bm \upsilon'),\mathbb T_c(\bm \omega',\bm \upsilon'))|
\geq C(\|\bm \omega'\|_{}^2+ \|\bm \upsilon'\|_{}^2),\quad \forall (\bm \omega',\bm \upsilon')\in\mathbf L^2(\Omega)\times\mathbf L^2(\Omega).
\end{eqnarray}
\end{lemma}
\begin{proof}First consider  the case $A^*>1$.
From the positive definite property of $\bm A$, a simple calculation shows that
\begin{align}\label{2.12ss}
|a((\bm \omega',\bm \upsilon'),\mathbb T_a(\bm \omega',\bm \upsilon'))|&=|(\bm A\curl\bm \omega',\curl\bm \omega')-(\curl\bm \upsilon',\curl(2\bm \omega'-\bm \upsilon')) |\nonumber\\
&\geq  A_* \|\curl\bm \omega'\|_{}^2+\|\curl\bm \upsilon'\|_{}^2 - 2|(\curl\bm \omega',\curl\bm \upsilon')| \nonumber\\
&\geq (A_*-\delta^{-1})\|\curl\bm \omega'\|_{}^2+\Big(1-\delta\Big)\|\curl\bm \upsilon'\|_{}^2\nonumber\\
&\geq C(\|\curl\bm \omega'\|_{}^2+ \|\curl\bm \upsilon'\|_{}^2).
\end{align}
Note that $A_*>1$, choosing $\delta \in \Big(\frac{1}{A_*},1\Big)$ leads to
   the $\mathbb T$-coercivity (\ref{tc1}) of  $a(\cdot,\cdot)$. 


For the case $N_*>1$, a similar argument shows that
\begin{align}\nonumber
|c((\bm \omega',\bm \upsilon'),\mathbb T_c(\bm \omega',\bm \upsilon'))|&=|(\bm N\bm \omega',\bm \omega')-(\bm \upsilon',2\bm \omega'-\bm \upsilon')|\nonumber\\
&\geq \Big( N_*-\frac{1}{\delta}\Big)\|\bm \omega'\|^2+(1-\delta)\|\bm \upsilon'\|^2
\end{align}
which leads to the $\mathbb T$-coercivity (\ref{tc2}) by choosing $\delta \in \Big(\frac{1}{N_*},1\Big)$.

For the case $A^*<1$, a simple calculation shows that
\begin{align}
|a((\bm \omega',\bm \upsilon'),\mathbb T_a(\bm \omega',\bm \upsilon'))|&=|(\bm A\curl\bm \omega',\curl(\bm \omega'-2\bm \upsilon'))
+(\curl\bm \upsilon',\curl \bm \upsilon') |\nonumber\\
&\geq  (\bm A\curl\bm \omega',\curl\bm \omega')+\|\curl\bm \upsilon'\|_{}^2  - 2|(\bm A\curl\bm \omega',\curl\bm \upsilon')| \nonumber\\
&\geq \Big(1-\frac{1}{\delta}\Big)(\bm A\curl\bm \omega',\curl\bm \omega')+\Big(1-A^*\delta\Big)\|\curl\bm \upsilon'\|_{}^2.\nonumber
\end{align}
Note that $A^*<1$, choosing $\delta \in \Big(1,\frac{1}{A^*}\Big)$   leads to
   the $\mathbb T$-coercivity (\ref{tc1}) of  $a(\cdot,\cdot)$. Similarly the $\mathbb T$-coercivity   of  $c(\cdot,\cdot)$ is also true  when $N^*<1$.
\end{proof}

For  $(\bm \omega',\bm \upsilon')\in\mathbf{L}^2(\Omega)\times \mathbf{L}^2(\Omega) $, consider the problem of finding $(\phi,\psi)\in \mathcal Q$ such that
\begin{align}\label{lap}
(\bm N\nabla  \phi,\nabla p')-(\nabla\psi,\nabla q')=(\bm N\bm \omega',\nabla p')-(\bm \upsilon',\nabla q'),~\forall (p',q')\in \mathcal Q.
\end{align}
From  the  $\mathbb T$-coercivity  property \eqref{tc2} and the Poincar\'e inequality in $H^1(\Omega)$ we have
\begin{align*}
|c((\nabla  p',\nabla q'),\mathbb T_c(\nabla p',\nabla q'))|\gtrsim \|\nabla  p'\|^2+\|\nabla q'\|^2\gtrsim \| p'\|_{1,\Omega}^2+\|q'\|_{1,\Omega}^2,~\forall (p',q')\in \mathcal Q.
\end{align*}
Hence the  problem \eqref{lap} is well-posed.
This infers the following direct sum decompositions
\begin{align}
\mathbf{L}^2(\Omega)\times \mathbf{L}^2(\Omega)&=\bm{\mathcal{L}}_{0}\bigoplus\{(\nabla p',\nabla q'):(p',q')\in \mathcal Q\},\label{sum1}\\
\bm{\mathcal H}&=\bm{\mathcal{H}}_{0}\bigoplus\{(\nabla p',\nabla q'):(p',q')\in \mathcal Q\},\label{sum}
\end{align}
where
\begin{align*}
\bm{\mathcal{L}}_{0}:=\{(\bm \omega',\bm \upsilon')\in \mathbf{L}^2(\Omega)\times\mathbf{L}^2(\Omega): b(( p', q'),(\bm \omega',\bm \upsilon'))=0,\quad\forall (p', q')\in \mathcal Q\}.
\end{align*}
  Due to the decomposition \eqref{sum}, the problem \eqref{Problem_Weak0} is equivalent to the following variational form: Find  $\lambda\in \mathbb C$ and nozero $(\bm \omega,\bm \upsilon) \in \bm{\mathcal{H}}_{0}$ such that
\begin{align}\label{Problem_Weakk}
a((\bm \omega,\bm \upsilon),(\bm \omega',\bm \upsilon'))&=\lambda c((\bm \omega,\bm \upsilon),(\bm \omega',\bm \upsilon')), \ \ \ \forall (\bm \omega',\bm \upsilon') \in \bm{\mathcal{H}}_{0}.
\end{align}
Note that $\lambda=0$ is not an eigenvalue of \eqref{Problem_Weakk}. In fact, if $\lambda=0$ in \eqref{Problem_Weakk} then
\[(\bm A\curl \bm \omega,\curl \bm \omega')-(\curl \bm \upsilon,\curl \bm \upsilon')=0,\quad \forall(\bm\omega',\bm \upsilon')\in \bm{\mathcal H}_0.\]
Due to the direct sum decomposition \eqref{sum}, the above inequality is also valid for $(\bm\omega',\bm \upsilon')=\mathbb T_a(\bm\omega,\bm \upsilon)$. From the $\mathbb T$-coercivity of $a(\cdot,\cdot)$ in Lemma 2.1,
 we see that
 $\|\curl \bm \omega\|^2+\|\curl \bm \upsilon\|^2\lesssim0,$  then  $\bm \omega=\bm \upsilon=0$ due to the Poincar\'e inequality postponed to Corollary \ref{Corollary 2.1}.

Furthermore the problem \eqref{Problem_Weak0} has another equivalent form:
Find $\lambda\in \mathbb C$, nozero $(\bm \omega,\bm \upsilon) \in \bm{\mathcal{H}}$ and $( \mathfrak p, \mathfrak q)\in \mathcal Q$ such that
\begin{subequations}\label{Problem_Weak}
\begin{align}
a((\bm \omega,\bm \upsilon),(\bm \omega',\bm \upsilon'))+&b(( \mathfrak p, \mathfrak q),(\bm \omega',\bm \upsilon'))=\lambda c((\bm \omega,\bm \upsilon),(\bm \omega',\bm \upsilon')), \ \ \ \forall (\bm \omega',\bm \upsilon') \in \bm{\mathcal{H}},\\
&b(( p', q'),(\bm \omega,\bm \upsilon))=0, \quad\forall (p',q')\in \mathcal Q.
\end{align}
\end{subequations}
Taking $(\bm \omega',\bm \upsilon')=(\nabla p',\nabla q')$ then the above $\mathfrak p$ and $\mathfrak q$ satisfy
\begin{align}
b(( \mathfrak p, \mathfrak q),(\nabla p',\nabla q'))=\lambda c( (\bm \omega,\bm \upsilon), (\nabla p',\nabla q'))=\lambda \overline{b( ( p', q'), (\bm \omega,\bm \upsilon))}=0, \quad\forall (p',q')\in \mathcal Q.
\end{align}
Then $\mathfrak p=\mathfrak q=0$ according to the  $\mathbb T$-coercivity  property \eqref{tc2}.
 The above problem is equivalent to finding $\lambda\in \mathbb{C}$, $(\bm \omega,\bm \upsilon)\in \bm{\mathcal{H}}_{0}$ and $(\mathfrak p,\mathfrak q)\in \mathcal Q$ such that
\begin{equation}\label{Problem_2s}
\left\{
\begin{array}{rcll}
{\curl}\big(\bm A\curl\bm \omega\big)+\bm N\nabla \mathfrak p&=& \lambda \bm N\bm \omega,&{\rm in}\ \Omega,\\
\curl^2 \bm \upsilon+\nabla \mathfrak q&=& \lambda \bm \upsilon, &{\rm in}\ \Omega,\\
\bm \nu\times(\bm \omega-\bm \upsilon)&=&0,&{\rm on}\ \partial \Omega,\\
\bm \nu\times (\bm A\curl\bm \omega-\curl \bm \upsilon)&=&0,&{\rm on}\ \partial \Omega.
\end{array}
\right.
\end{equation}

\begin{lemma}[Lemma 3.50 in \cite{monk1}]\label{lemmap1}For  any $\bm\chi_1,\bm\chi_2\in \Hcurl{\Omega}$ satisfying $\div\bm\chi_1,\div\bm\chi_2\in \mathbf L^2(\Omega)$ and $\bm\nu\times \bm\chi_1=\bm\nu\cdot \bm\chi_2=0$ on $\partial\Omega$ there is $r\in(1/2,1]$ only dependent on the shape of the domain such that
\begin{align*}
  \|\bm\chi_1\|_{r,\Omega}&\lesssim \|\div\bm\chi_1\|_{}+\|\curl\bm\chi_1\|_{},\\
  \|\bm\chi_2\|_{r,\Omega}&\lesssim \|\div\bm\chi_2\|_{}+\|\curl\bm\chi_2\|_{}.
\end{align*}
\end{lemma}
If $({\bm w}',{\bm v}')\in \bm{\mathcal{H}}_{0}$ and $\bm N=n\bm I$ with $n\in   W^{1,\infty}(\Omega)$ then $0=\div\bm v'=\div(n\bm w')=n\div\bm w'+\nabla n\cdot\bm w'$ in $\Omega$ and $(n{\bm w}'-{\bm v}')\cdot\bm \nu=0$ on $\partial\Omega$ which infers that $\div\bm w'\in\mathbf L^2(\Omega)$; noting that $\curl(n\bm w')=\nabla n\times \bm w'+n\curl\bm w'$ we have $n{\bm w}',{\bm w}',{\bm v}'\in \Hcurl{\Omega}$. According to Lemma \ref{lemmap1}, if  $\bm N=n\bm I$ with $n\in   W^{1,\infty}(\Omega)$ then we have the fact that $({\bm w}',{\bm v}')\in \bm{\mathcal{H}}_{0}$ implies $ {\bm w}'- {\bm v}'\in \mathbf H^r(\Omega)$,  $n{\bm w}'-{\bm v}'\in \mathbf H^r(\Omega)$ and
\begin{align}\label{subset0}
\|(\bm w',\bm v')\|_{r,\Omega}\lesssim \|\curl({\bm w}'- {\bm v}')\|+\|\curl(n{\bm w}'- {\bm v}')\|+\|{\bm w}'\|,
\end{align}
that is
\begin{align}\label{subset}
\bm{\mathcal{H}}_{0}\subset \mathbf H^r(\Omega)\times \mathbf H^r(\Omega).
\end{align}
In particular, when $\bm N=n\bm I$ with a positive constant $n$ we have
 \begin{align}\label{subset00}
\|(\bm w',\bm v')\|_{r,\Omega}\lesssim \|\curl({\bm w}'- {\bm v}')\|+\|\curl(n{\bm w}'- {\bm v}')\|.
\end{align}

\begin{corollary}[Poinc{a}r\'e inequality]\label{Corollary 2.1}For any $({\bm w}_{}',{\bm v}_{}')\in \bm{\mathcal{H}}_{0}$ it holds
\begin{align}\label{po}
    \|({\bm w}_{}',{\bm v}_{}')\|_{}\lesssim \|(\curl\bm w_{}',\curl\bm v_{}')\|_{}.
  \end{align}
\end{corollary}
\begin{proof}We represent $\bm{\mathcal{H}}_{0}^S$ as  the special $\bm{\mathcal{H}}_{0}$ in the case of   $\bm N=2\bm I$.  Obviously the Poinc{a}r\'e inequality \eqref{po} is valid for $({\bm w}_{}',{\bm v}_{}')\in\bm{\mathcal{H}}_{0}^S$ due to \eqref{subset00}. Let  $(\bm{w}' ,\bm{v}')\in\bm{\mathcal{H}}_0$.
  Using the  decomposition \eqref{sum} of  $\bm{\mathcal{H}}_0^S$, we have
$(\bm{w}',\bm{v}')=(\bm{w}^S,\bm{v}^S)+(\nabla p,\nabla q)$  with some $(\bm w^S,\bm v^S)\in \bm {\mathcal H}^S_{0}$ and some $( p, q)\in \mathcal Q_{}$.
   Using the fact that  $(\boldsymbol{w}'_{},\boldsymbol{v}'_{}) $    is in  $\bm{\mathcal{H}}_0$   we have
 $$c((\bm w',\bm v'),(\nabla p',\nabla q'))=0,~\forall ( p',  q')\in \mathcal Q_{}.$$
It follows that
       \begin{align}\label{cc1}
    c((\bm w',\bm v'),\mathbb T_c(\bm w',\bm v'))
=c((\bm w',\bm v'),\mathbb T_c(\bm w^S,\bm v^S)).
    \end{align}
Hence
 $$\left\|(\bm w',\bm v')\right\|_{} \lesssim\left\|(\bm w^S,\bm v^S)\right\|_{ }\mathop{\lesssim}^{\text{by }\eqref{subset00}} \|(\curl\bm w^S,\curl\bm v^S)\|_{}= \|(\curl\bm w_{}',\curl\bm v_{}')\|_{}.$$
This concludes the proof.
\end{proof}

\begin{lemma}\label{lemmadense}$\mathbf H:=\{(\bm{w}',\bm{v}')\in \mathbf C^\infty(\Omega)\times\mathbf C^\infty(\Omega) : \bm{w}'-\bm{v}'\in\Hcurltr{\Omega}\}$ is dense in $\bm{\mathcal{H}}$ and
$Q=:\{(p',q')\in C^\infty(\Omega)\times C^\infty(\Omega) :\int_{\partial\Omega}p'=\int_{\partial\Omega}q'=0\text{ and } p'-q'\in H^1_0(\Omega)\}$ is dense in $\mathcal Q$.
\end{lemma}
\begin{proof}Introduce the orthogonal decomposition $\bm{\mathcal{H}}=\overline{\mathbf H}\bigoplus\overline{\mathbf H}^\bot$. It suffices to prove
$\overline{\mathbf H}^\bot=\{\bm 0\}$. Let $(\widetilde{\bm w},\widetilde{\bm v})\in \bm{\mathcal{H}}$ be such that
\begin{align} \label{or}
(\curl \widetilde{\bm w},\curl \bm w')+(\curl \widetilde{\bm v},\curl \bm v')+(\widetilde{\bm w},\bm w')+(\widetilde{\bm v}, \bm v')=0,\quad \forall (\bm w',\bm v')\in \mathbf H.
\end{align}
Then $\curl^2 \widetilde{\bm w}+\widetilde{\bm w}=\curl^2 \widetilde{\bm v}+\widetilde{\bm v}=0$ in $\Omega$, which together with $\widetilde{\bm{w}}-\widetilde{\bm{v}}\in\Hcurltr{\Omega}$ yields $\widetilde{\bm{w}}=\widetilde{\bm{v}}$ in $\Omega$. Taking ${\bm{w}}'= {\bm{v}}'$ in \eqref{or},  we have $$(\curl \widetilde{\bm w},\curl \bm w')+(\widetilde{\bm w},\bm w')=0,~\forall \bm w'\in\Hcurl{\Omega}.$$ Hence $\widetilde{\bm w}=0$ in $\Omega.$ This means $\overline{\mathbf H}^\bot=\{\bm 0\}$. The similar argument as above can show that $\overline{\mathcal Q}^\bot=\{\bm 0\}$. Then $\mathcal Q=\overline{ Q}$.
\end{proof}

\section{Edge element discretizations and error analysis}
\subsection{Edge element discretizations and  source problems}
Let $\mathcal T_h=\{K\}$ be a tetrahedral mesh of $\Omega$.
The $k$($k\ge0$) order edge element of the first family \cite{nedelec} generates the space
$$\mathbf V_h:=\{\bm{v}'\in\mathbf H_0(\mathrm{curl},\Omega):\bm{v}'|_K\in [P_k(K)]^3\bigoplus\bm{x}\times [\tilde{P}_k(K)]^3,~\forall K\in \mathcal T_h\},$$
where $P_k(K)$ is the polynomial space of  degree less than or equal to $ k$ on $K$, ${\tilde{P}}_k(K)$ is the
homogeneous polynomial space of  degree $ k$  on $K$, and $\bm x =(x_1,x_2,x_3)^T$. Hence we can define the finite element spaces associated with $\bm{\mathcal{H}}$ and $Q$ as follows:
\begin{align*}
\bm{\mathcal{H}}_h&:=\bm{\mathcal{H}}\cap (\mathbf V_h\times \mathbf V_h),\\
\mathcal Q_h&:=\{(p,q)\in \mathcal Q: p|_K,q|_K\in {P}_{k+1}(K),~ \forall K\in \mathcal T_h\},\\
\bm{\mathcal{H}}_{0,h}&:=\{(\bm \omega',\bm \upsilon')\in \bm{\mathcal{H}}_h: b(( p', q'),(\bm \omega',\bm \upsilon'))=0,\quad\forall (p', q')\in \mathcal Q_h\}.
\end{align*}

Similar as \eqref{sum}, the following discrete direct sum decomposition holds:
\begin{align}\label{dsum}
\bm{\mathcal{H}}_h=\bm{\mathcal{H}}_{0,h}\bigoplus\{(\nabla p',\nabla q'):(p',q')\in \mathcal Q_h\}.
\end{align}

The discrete variational form of (\ref{Problem_Weak}) can be defined as follows:
Find   $\lambda_h\in \mathbb C\backslash\{0\}$ and nonzero $(\bm \omega_h,\bm \upsilon_h) \in \bm{\mathcal{H}}_{h}$ such that
\begin{align}
a((\bm \omega_h,\bm \upsilon_h),(\bm \omega',\bm \upsilon'))&=\lambda_h c((\bm \omega_h,\bm \upsilon_h),(\bm \omega',\bm \upsilon')), \ \ \ \forall (\bm \omega',\bm \upsilon') \in \bm{\mathcal{H}}_{h}.
\end{align}

The above variational form is equivalent to  finding
 $\lambda_h\in \mathbb C$, nozero $(\bm \omega_h,\bm \upsilon_h) \in \bm{\mathcal{H}}_h$ and $( \mathfrak p_h, \mathfrak q_h)\in \mathcal Q_h$ such that
\begin{subequations}\label{Problem_Weakd}
\begin{align}
a((\bm \omega_h,\bm \upsilon_h),(\bm \omega',\bm \upsilon'))+&b(( \mathfrak p_h, \mathfrak q_h),(\bm \omega',\bm \upsilon'))=\lambda_h c((\bm \omega_h,\bm \upsilon_h),(\bm \omega',\bm \upsilon')), \ \ \ \forall (\bm \omega',\bm \upsilon') \in \bm{\mathcal{H}}_h,\\
&b(( p', q'),(\bm \omega_h,\bm \upsilon_h))=0, \quad\forall (p',q')\in \mathcal Q_h.
\end{align}
\end{subequations}
Taking $(\bm \omega',\bm \upsilon'):=(\nabla p',\nabla q')$ in \eqref{Problem_Weakd} we have $c(( \nabla \mathfrak p_h, \nabla \mathfrak q_h),(\nabla \mathfrak p', \nabla \mathfrak q'))=0,~\forall (\mathfrak p',  \mathfrak q')\in \mathcal Q_h$. Then $\mathfrak p_h=\mathfrak q_h=0$ due to \eqref{tc2}.

In order to analyze the convergence of the discretization \eqref{Problem_Weakd}, we need to consider the source problems associated with \eqref{Problem_Weak} and \eqref{Problem_Weakd}, respectively.

 Define the following Maxwell's transmission problem: Find $(\bm w,\bm v)\in \bm{\mathcal{H}}_{0}$ and $(p,q)\in Q$ such that
\begin{equation}\label{Problem_2ss}
\left\{
\begin{array}{rcll}
{\curl}\big(\bm A\curl\bm w\big)+\bm N\nabla p&=& \bm N\bm f,&{\rm in}\ \Omega,\\
\curl^2 \bm v+\nabla q&=& \bm g, &{\rm in}\ \Omega,\\
\bm \nu\times(\bm w-\bm v)&=&0,&{\rm on}\ \partial \Omega,\\
\bm \nu\times (\bm A\curl\bm w-\curl \bm v)&=&0,&{\rm on}\ \partial \Omega.
\end{array}
\right.
\end{equation}

The associated variational form of (\ref{Problem_2ss}) can be defined as follows:
Find $(\bm w,\bm v) \in \bm{\mathcal{H}}$ and $( p, q)\in\mathcal Q$ such that
\begin{subequations}\label{ProblemWeak}
\begin{align}
a((\bm w,\bm v),(\bm w',\bm v'))+&b(( p, q),(\bm w',\bm v'))=  c((\bm f,\bm g),(\bm w',\bm v')), \ \ \ \forall (\bm w',\bm v') \in \bm{\mathcal{H}},\\
&b(( p', q'),(\bm w,\bm v))=0, \quad\forall (p',q')\in\mathcal Q,
\end{align}
\end{subequations}
which is equivalent to two problems as follows: Find
   $( p, q)\in\mathcal Q$ such that
\begin{align}\label{Problem_Weak2}
c((\nabla p,\nabla q),(\nabla p',\nabla q'))=  c((\bm f,\bm g),(\nabla p',\nabla q')) , \quad \forall (p',q') \in\mathcal Q,
\end{align}
and find
 $(\bm w,\bm v) \in \bm{\mathcal{H}}_0$ such that
\begin{align}\label{Problem_Weak1}
a((\bm w,\bm v),\mathbb T_a(\bm w',\bm v'))=  c((\bm f,\bm g),\mathbb T_a(\bm w',\bm v'))-c((\nabla p,\nabla q),\mathbb T_a(\bm w',\bm v')), \quad \forall (\bm w',\bm v') \in \bm{\mathcal{H}}_0,
\end{align}
or  \begin{align}
a((\bm w,\bm v),(\bm w',\bm v'))=  c((\bm f,\bm g),(\bm w',\bm v')), \quad \forall (\bm w',\bm v') \in \bm{\mathcal{H}}_{0}.
\end{align}
The above equivalence is obtained by the decomposition \eqref{sum} and the fact that
$$c((\nabla p,\nabla q),\mathbb T_a(\nabla p',\nabla q'))=c((\bm f,\bm g),\mathbb T_a(\nabla p',\nabla q')),~\forall(p',q')\in\mathcal Q.$$
The $\mathbb T$-coercivity properties \eqref{tc1} and \eqref{tc2} implies that the problems \eqref{Problem_Weak2} and  \eqref{Problem_Weak1} are uniquely solvable, respectively. This leads to the well-posedness of the above problem   and the stability
\begin{align}\label{sta}
\|(\bm w,\bm v)\|_{\curl}\lesssim \|(\bm f,\bm g)\|.
\end{align}

The discrete variational form of (\ref{Problem_2ss}) can be defined as follows:
Find $(\bm w_h,\bm v_h) \in \bm{\mathcal{H}}_h$ and $( p_h, q_h)\in\mathcal Q_h$ such that
\begin{subequations}\label{Problem_Weakdd}
\begin{align}
a((\bm w_h,\bm v_h),(\bm w',\bm v'))+&b(( p_h, q_h),(\bm w',\bm v'))=  c((\bm f,\bm g),(\bm w',\bm v')), \ \ \ \forall (\bm w',\bm v') \in \bm{\mathcal{H}}_h,\\
&b(( p', q'),(\bm w_h,\bm v_h))=0, \quad\forall (p',q')\in\mathcal Q_h,
\end{align}
\end{subequations}
which is equivalent to two problems as follows: Find
   $( p_h, q_h)\in \mathcal Q_h$ such that
\begin{align}\label{Problem_Weak2d}
c((\nabla p_h,\nabla q_h),(\nabla p',\nabla q'))=  c((\bm f,\bm g),(\nabla p',\nabla q')) , \quad \forall (p',q') \in \mathcal Q_h,
\end{align}
and find
 $(\bm w_h,\bm v_h) \in \bm{\mathcal{H}}_{0,h}$ such that
\begin{align}\label{Problem_Weak1d}
a((\bm w_h,\bm v_h),\mathbb T_a(\bm w',\bm v'))=  c((\bm f,\bm g),\mathbb T_a(\bm w',\bm v'))-c((\nabla p_h,\nabla q_h),\mathbb T(\bm w',\bm v')), \quad \forall (\bm w',\bm v') \in \bm{\mathcal{H}}_{0,h},
\end{align}
or \begin{align}
a((\bm w_h,\bm v_h),(\bm w',\bm v'))=  c((\bm f,\bm g),(\bm w',\bm v')), \quad \forall (\bm w',\bm v') \in \bm{\mathcal{H}}_{0,h}.
\end{align}
The above equivalence is obtained by  the decomposition \eqref{dsum}
and the fact that
$$c((\nabla p_h,\nabla q_h),\mathbb T_a(\nabla p',\nabla q'))=c((\bm f,\bm g),\mathbb T_a(\nabla p',\nabla q')),~\forall(p',q')\in\mathcal Q_h.$$

\subsection{Error analysis via discrete compactness}
Let $\bm r_h$  be the edge element interpolation operator on $\mathbf V_h$ then we have the following.
\begin{lemma}[Theorem 5.41 in \cite{monk1}] We have the following estimates.

 (1) If  $\bm{u} \in \mathbf{H}^{s}(K)$ for  $1 / 2<s \leq 1$  and $ \left.\curl\bm{u}\right|_{K} \in \curl([P_k(K)]^3\bigoplus\bm{x}\times [\tilde{P}_k(K)]^3)$, then
\begin{align}\label{i1}
\left\|\bm{u}-\bm{r}_{h} \bm{u}\right\|_{0, K} \lesssim h_{K}^s\|\bm{u}\|_{s, K}+h_{K}\|\mathbf{c u r l} \bm{u}\|_{0, K}.
\end{align}
(2) If $\bm{u} \in  \mathbf H^s(\curl,\Omega)$ for  $1 / 2<s \leq k+1$, then
\begin{align}\label{i2}
\left\|\bm{u}-\bm{r}_{h} \bm{u}\right\|_{\Hcurl{K})} \lesssim h^{s}(\|\bm{u}\|_{s,K}+\|\curl\bm{u}\|_{s,K}).
\end{align}
\end{lemma}

Like the Hodge mapping in \cite{monk1, hiptmair}, any element in $\bm{\mathcal{H}}_{0,h}$ has its approximation in $\bm{\mathcal{H}}_{0}$.
\begin{lemma}\label{Lemma 3.2}Assume  $\bm N=n\bm I$ with $n\in   W^{1,\infty}(\Omega)$. For any $(\bm{w}'_h,\bm{v}'_h)\in \bm{\mathcal{H}}_{0,h}$, there is $(\bm{w}^0,\bm{v}^0)\in \bm{\mathcal{H}}_{0}$ satisfying
$\curl\bm{w}^0=\curl\bm{w}'_h$ and  $\curl\bm{v}^0=\curl\bm{v}'_h$ such that
\begin{align}\label{p1}
&\|\bm{v}^0\|_{r,\Omega}+\|\bm{w}^0\|_{r,\Omega}\lesssim \|\curl\bm{v}'_h\|+\|\curl\bm{w}'_h\|,\\
&\|\bm{w}^0-\bm{w}'_h\|+\|\bm{v}^0-\bm{v}'_h\|\lesssim h^{r} ( \|\curl {\bm{w}}'_h \|+\|\curl{\bm { v}}'_h\|_{}).\label{p2}
\end{align}
\end{lemma}
\begin{proof}
Introduce the auxiliary problem: Find $(\bm{w}^0,\bm{v}^0)\in  \boldsymbol{\mathcal{H}} $ and $ (\rho,\varrho)\in\mathcal Q  $ such that
\begin{subequations}\label{3.7}
\begin{align}
 &(\curl\bm{w}^0,\curl{\bm w}')+(\curl\bm{v}^0,\curl{\bm v}') +b((\rho, \varrho),({\bm{w}'},{\bm{v}'} ))\nonumber\\
 &\quad= (\curl\bm{w}'_h,\curl{\bm w}')+(\curl\bm{v}'_h,\curl{\bm v}'),\quad\forall (\bm{w}',\bm{v}')\in\boldsymbol{\mathcal{H}},\\
&b((\rho', \varrho'),(\bm{w}^0,\bm{v}^0 ))=0,\quad\forall (\rho',\varrho')\in\mathcal Q.\label{3.7b}
\end{align}
\end{subequations}
The above equation is well-posed. In fact, choosing $(\bm{w}',\bm{v}')$ as $(\nabla p',\nabla q')$ with $(p',q')\in\mathcal Q$ and using the $\mathbb T$-coercivity property (\ref{tc2}) we get $\rho=\varrho=0$; then there are $t,u\in \mathcal Q$ such that $\bm{w}^0=\bm{w}'_h+\nabla t$ and
$\bm{v}^0=\bm{v}'_h+\nabla u$,
 which together with  \eqref{3.7b} yields
\begin{align}
c((\nabla \rho', \nabla\varrho'),(\nabla t,\nabla u ))=-b((\rho', \varrho'),(\bm{w}'_h, \bm{v}'_h )),\quad\forall (\rho',\varrho')\in\mathcal Q
\end{align}
Using again the $\mathbb T$-coercivity property (\ref{tc2}), $t$ and $u$ are uniquely solvable.

According to the imbedding relation \eqref{subset}, $(\bm{w}^0,\bm{v}^0)\in\boldsymbol{\mathcal{H}}_0$ implies that  $\bm{w}^0,\bm{v}^0\in \mathbf H^r(\Omega)$, which together with Lemma \ref{lemmap1}, Corollary \ref{Corollary 2.1}  and
 $\|\curl\bm{v}^0\|+\|\curl\bm{w}^0\| = \|\curl\bm{v}'_h\|+\|\curl\bm{w}'_h\|$ yields \eqref{p1}.  According to Lemma 5.40 in \cite{monk1}, then  $\curl\bm r_h\bm{w}^0=\curl\bm{w}'_h$ and $\curl\bm r_h\bm{v}^0=\curl\bm{v}'_h$. Noting that $\bm\nu\times(\bm{w}'_h-\bm{v}'_h)=\bm\nu\times(\bm r_h\bm{w}^0-\bm{w}^0)=0$ on $\partial\Omega$,  there is $t_h,u_h\in\mathcal Q_h$ such that $\bm r_h\bm{w}^0=\bm{w}'_h+\nabla t_h$ and $ \bm r_h\bm{v}^0= \bm{v}'_h+\nabla u_h$.
Then
\begin{align*}
 c((\bm{w}^0-\bm{w}'_h,\bm{v}^0-\bm{v}'_h),\mathbb T_c(r_h\bm{w}^0-\bm{w}'_h,r_h\bm{v}^0-\bm{v}'_h))=0
\end{align*}
which together with \eqref{i1} yields that
\begin{align*}
&|c((\bm{w}^0-\bm{w}'_h,\bm{v}^0-\bm{v}_h'),\mathbb T_c(\bm{w}^0-\bm{w}'_h,\bm{v}^0-\bm{v}_h'))|\\
&\quad=|c((\bm{w}^0-\bm{w}'_h,\bm{v}^0-\bm{v}'_h),\mathbb T_c(\bm{w}^0-r_h\bm{w}^0,{\bm v}^0-r_h\bm{v}^0))|\\
&\quad\lesssim (\|\bm{w}^0-\bm{w}'_h\|+\|\bm{v}^0-\bm{v}'_h\|)h^{r} (\|\bm{w}^0\|_r+ \|\curl {\bm{w}}'_h \|+\|{\bm { v}^0}\|_{r}+\|\curl{\bm { v}}'_h\|_{}).
\end{align*}
Then
\begin{align*}
\|\bm{w}^0-\bm{w}'_h\|+\|\bm{v}^0-\bm{v}'_h\|\lesssim h^{r} (\|\bm{w}^0\|_r+ \|\curl {\bm{w}}'_h \|+\|{\bm { v}}^0\|_{r}+\|\curl{\bm { v}}'_h\|_{}).
\end{align*}
This together with \eqref{p1} yields \eqref{p2}.
\end{proof}

\begin{corollary}[Discrete Poinc{a}r\'e inequality]For any $({\bm w}_{}',{\bm v}_{}')\in \bm{\mathcal{H}}_{0,h}$ it holds
\begin{align}\label{dpi}
    \|({\bm w}_{}',{\bm v}_{}')\|_{}\lesssim \|(\curl\bm w_{}',\curl\bm v_{}')\|_{}.
  \end{align}
\end{corollary}
\begin{proof}
According to Lemma \ref{Lemma 3.2}, the estimate \eqref{dpi} is valid when  $\bm N=n\bm I$ with $n\in   W^{1,\infty}(\Omega)$. The conclusion can be proved using the similar argument as in Corollary \ref{Corollary 2.1} and the decomposition \eqref{dsum}.
\end{proof}
In what follows, we represent $\bm{\mathcal{H}}_{0,h}^S$ and  $\bm{\mathcal{H}}_{0}^S$ as  the special $\bm{\mathcal{H}}_{0,h}$ and $\bm{\mathcal{H}}_{0}$ in the case of   $\bm N=2\bm I$, respectively.
Let ${\mathcal{M}}$ be a sequence of the mesh size  $h_i$ with $h_i$ converging to 0 as $i\to0$.
\begin{lemma}(Discrete compactness property) \label{dcp}  Any sequence $\{({\bm w}'_{h},\bm v'_{h})\}_{h\in\mathcal{M}}$ with $({\bm w}'_{h},\bm v'_{h})\in \bm{\mathcal{H}}_{0,h}$ that is uniformly bounded w.r.t $\|\cdot\|_{\bm{\mathcal{H}}}$ contains a subsequence that converges strongly in $\mathbf L^2(\Omega)\times \mathbf L^2(\Omega)$.
\end{lemma}
\begin{proof}
  Let $\{(\bm w'_{h  },\bm v'_{h  })\}_{h\in\mathcal{M}}\subset \bm{\mathcal{H}}_{0,h}^S$  with $\|(\bm w'_{h  },\bm v'_{h  })\|_{\bm{\mathcal{H}}}<M$ for a positive constant $M$. It is trivial to  assume that the sequence $h_i\in\mathcal{M}$ satisfies $h_i\rightarrow0$    as $i\to\infty$. According to Lemma \ref{Lemma 3.2}, there is $({\bm w}_{h_i}^0,{\bm v}_{h_i}^0)\in\bm{\mathcal{H}}^S_{0}$ such that $\|({\bm w}_{h_i}^0- {\bm w}'_{h_i},{\bm v}_{h_i}^0- {\bm v}'_{h_i})\|\to 0$ as $i\to\infty$.
  Note that  from \eqref{p1}  we deduce
  \begin{align*}
    \|({\bm w}_{h_i}^0,{\bm v}_{h_i}^0)\|_{\bm{\mathcal{H}}}\lesssim \|(\bm w'_{h_i},\bm v'_{h_i})\|_{\bm{\mathcal{H}}}.
  \end{align*}
This means that  $\{({\bm w}_{h_i}^0,{\bm v}_{h_i}^0)\}$ is uniformly bounded in $\bm{\mathcal{H}}^S_{0}$. Since $\bm{\mathcal{H}}^S_{0}$ is compactly imbedded into $\mathbf L^2(\Omega)\times \mathbf L^2(\Omega)$, there is a subsequence of $\{({\bm w}_{h_i}^0,{\bm v}_{h_i}^0)\}$ converging to some $(\bm w_0,\bm v_0)$ in $\mathbf L^2(\Omega)\times \mathbf L^2(\Omega)$. Hence   a subsequence of $\{({\bm w}'_{h_i},{\bm v}'_{h_i})\}$ will converge to  $({\bm w}_0,\bm v_0)$ in $\mathbf L^2(\Omega)\times \mathbf L^2(\Omega)$ as well. At this time we have proved the discrete compactness property of $ \bm{\mathcal{H}}_{0,h}^S$.

 Next we shall prove the discrete compactness property of $ \bm{\mathcal{H}}_{0,h}$.   Let $\{(\bm w'_{h_i  },\bm v'_{h_i  })\} \subset \bm{\mathcal{H}}_{0,h}$  with $\|(\bm w'_{h_i  },\bm v'_{h_i  })\|_{\bm{\mathcal{H}}}<M$ for a positive constant $M$.
 According to the decomposition \eqref{dsum},
  $(\bm w'_{h_i  },\bm v'_{h_i  })=(\bm w_{h_i  }^S,\bm v_{h_i  }^S)+(\nabla p^S_{h_i},\nabla q^S_{h_i})$ with $(\bm w_{h_i  }^S,\bm v_{h_i  }^S)\in \bm {\mathcal H}_{0,h}^S$ and $( p^S_{h_i}, q^S_{h_i})\in \mathcal Q_{h_i}$. 
     Let $(\bm w_{h_i  }^S,\bm v_{h_i  }^S)$ converge strongly to  $({\bm w}^S,\bm v^S)$ in  $\mathbf L^2(\Omega)\times \mathbf L^2(\Omega)$. 
       According to the decomposition \eqref{sum1},
  $(\bm w^S,\bm v^S)=(\bm w^0,\bm v^0)+(\nabla p^0,\nabla q^0)$ with $(\bm w^0,\bm v^0)\in \bm {\mathcal L}_{0}$ and $( p^0, q^0)\in \mathcal Q_{}$. 
Using the fact that $$c((\bm w'_{h_i  },\bm v'_{h_i  }),(\nabla p'_{h_i},\nabla q'_{h_i})=c((\bm w^0_{},\bm v^0_{}),(\nabla p'_{h_i},\nabla q'_{h_i})=0,~\forall ( p'_{h_i},  q'_{h_i})\in \mathcal Q_{h_i},$$
    we have
       \begin{align}\label{cc4}
   &\quad c((\bm w^0-\bm w'_{h_i  },\bm v^0-\bm v'_{h_i  }),\mathbb T_c(\bm w^0-\bm w'_{h_i  },\bm v^0-\bm v'_{h_i  }))\nonumber\\
   &=c((\bm w^0-\bm w'_{h_i  },\bm v^0-\bm v'_{h_i  }),\mathbb T_c(\bm w^0-\bm w_{h_i  }^S,\bm v^0-\bm v_{h_i  }^S)-\mathbb T_c(\nabla p^S_{h_i}, \nabla q^S_{h_i}))\nonumber\\
&=c((\bm w^0-\bm w'_{h_i  },\bm v^0-\bm v'_{h_i  }),\mathbb T_c(\bm w^S-\bm w_{h_i  }^S,\bm v^S-\bm v_{h_i  }^S)+\mathbb T_c(\nabla (p'_{h_i}-p^0), \nabla (q'_{h_i}-q^0))),\nonumber\\
&\quad\forall ( p'_{h_i},  q'_{h_i})\in \mathcal Q_{h_i}.
    \end{align}
This gives
 \begin{align}\label{cc5}
   &\|(\bm w^0-\bm w'_{h_i  },\bm v^0-\bm v'_{h_i  })\|
   \lesssim\|(\bm w^S-\bm w_{h_i  }^S,\bm v^S-\bm v_{h_i  }^S)\|+\|(\nabla (p'_{h_i}-p^0), \nabla (q'_{h_i}-q^0))\|,\nonumber\\
&\quad\forall ( p'_{h_i},  q'_{h_i})\in \mathcal Q_{h_i}.
    \end{align}
This together with Lemma \ref{lemmadense} infers that   $(\bm w'_{h_i  },\bm v'_{h_i  })$   converges strongly to $(\bm w^0,\bm v^0)$ in $\mathbf L^2(\Omega)\times \mathbf L^2(\Omega)$.
\end{proof}

Next we shall establish the error bound of the finite element discretization \eqref{Problem_Weakdd} for approximating \eqref{ProblemWeak} as follows.
\begin{theorem}\label{theorem3.1}The following error estimate holds:
\begin{align}\label{pri}
\|(\bm w_h-\bm w,\bm v_h-\bm v)\|_{\bm{\mathcal{H}}}\lesssim& \inf_{(p',q')\in \mathcal Q_h}\|(p'-p,q'-q)\|_\mathcal Q+\inf_{(\bm w',\bm v')\in \bm{\mathcal{H}}_h}\|(\bm w-\bm w',\bm v-\bm v'))\|_{\bm{\mathcal{H}}}.
\end{align}
\begin{proof} We define $( \widetilde{\bm w}_h,\widetilde{\bm v}_h)\in \bm{\mathcal{H}}_h$ such that
 \begin{align}
\tilde a((\bm w-\widetilde{\bm w}_h,\bm v-\widetilde{\bm v}_h),(\bm w',\bm v'))=0,~\forall (\bm w',\bm v')\in \bm{\mathcal{H}}_h.
\end{align}
where
 \begin{align}
\tilde a((\bm w,\bm v),(\bm w',\bm v'))&:=\big(\bm N\curl \bm w,\curl \bm w'\big)+\big(\bm N\bm w,\bm w'\big)
-\big(\curl \bm v,\curl \bm v'\big)-\big(\bm v,\bm v'\big).
\end{align}
Using the similar argument as in Lemma 2.1, we see that
 \begin{align}
|\tilde a((\bm w',\bm v'),\mathbb T_c(\bm w',\bm v'))|\gtrsim \|(\bm w',\bm v')\|_{\bm{\mathcal{H}}}, ~\forall(\bm w',\bm v')\in\bm{\mathcal{H}}.
\end{align}
Therefore such  $(\widetilde{\bm w}_h,\widetilde{\bm v}_h)$ is well defined. It is easy to check that $(\widetilde{\bm w}_h,\widetilde{\bm v}_h)\in \bm{\mathcal{H}}_{0,h}$ and
 \begin{align}\label{proj}
\|(\bm w-\widetilde{\bm w}_h,\bm v-\widetilde{\bm v}_h)\|_{\bm{\mathcal{H}}_{}}\lesssim \inf_{(\bm w',\bm v')\in \bm{\mathcal{H}}_h}\|(\bm w-\bm w',\bm v-\bm v'))\|_{\bm{\mathcal{H}}}.
\end{align}
Choose  arbitrary $(\bm w',\bm v') \in \bm{\mathcal{H}}_h.$
The combination of \eqref{ProblemWeak} and \eqref{Problem_Weakdd} gives
\begin{align}
a((\bm w_h-\bm w,\bm v_h-\bm v),(\bm w',\bm v'))=  -c((\nabla (p_h-p),\nabla (q_h-q)),(\bm w',\bm v')).
\end{align}
Recalling the $\mathbb T$-coercivity \eqref{tc1} of $a(\cdot,\cdot)$, we reach
\begin{align}
&\quad\|(\curl(\bm w_h-\bm w),\curl(\bm v_h-\bm v))\|_{}^2\lesssim |a((\bm w_h-\bm w,\bm v_h-\bm v),\mathbb T_a(\bm w_h-\bm w,\bm v_h-\bm v))|\nonumber\\
&= |a((\bm w_h-\bm w,\bm v_h-\bm v),\mathbb T_a(\widetilde{\bm w}_h-\bm w,\widetilde{\bm v}_h-\bm v))-c((\nabla (p_h-p),\nabla (q_h-q)),\mathbb T_a(\bm w_h-\widetilde{\bm w}_h,\bm v_h-\widetilde{\bm v}_h))|.\nonumber
\end{align}
 Recalling the discrete Poincar\'e inequality \eqref{dpi}, this leads to
\begin{align}
&\quad\|(\curl(\bm w_h-\bm w),\curl(\bm v_h-\bm v))\|_{}^2\nonumber\\
&\lesssim  \|(\curl(\bm w_h-\bm w),\curl(\bm v_h-\bm v))\|_{}\|\curl(\widetilde{\bm w}_h-\bm w),\curl(\widetilde{\bm v}_h-\bm v)\|_{}\nonumber\\
&+\|(p_h-p,q_h-q)\|_\mathcal Q\|(\curl(\bm w_h-\bm w+\bm w-\widetilde{\bm w}_h),\curl(\bm v_h-\bm v+\bm v-\widetilde{\bm v}_h))\|.
\end{align}
By the Young's equality   we get
\begin{align}\label{w1}
\|(\curl(\bm w_h-\bm w),\curl(\bm v_h-\bm v))\|\lesssim \|(p_h-p,q_h-q)\|_\mathcal Q+\|(\bm w-\widetilde{\bm w}_h,\bm v-\widetilde{\bm v}_h))\|_{\bm{\mathcal{H}}},
\end{align}
The combination of \eqref{Problem_Weak2}, \eqref{Problem_Weak2d}  and the $\mathbb T$-coercivity \eqref{tc2} of $c(\cdot,\cdot)$ yields
\begin{align}\label{w2}
\|(p_h-p,q_h-q)\|_\mathcal Q\lesssim \inf_{(p',q')\in\mathcal Q}\|(p'-p,q'-q)\|_\mathcal Q.
\end{align}
It can be inferred from the estimates \eqref{w1}, \eqref{w2} and \eqref{proj}
\begin{align}\label{pri0}
\|(\curl(\bm w_h-\bm w),\curl(\bm v_h-\bm v))\|\lesssim& \inf_{(p',q')\in \mathcal Q_h}\|(p'-p,q'-q)\|_\mathcal Q+\inf_{(\bm w',\bm v')\in \bm{\mathcal{H}}_h}\|(\bm w-\bm w',\bm v-\bm v'))\|_{\bm{\mathcal{H}}}.
\end{align}
The triangular inequality and   the discrete Poincar\'e inequality \eqref{dpi} implies
\begin{align*}
\|(\bm w_h-\bm w,\bm v_h-\bm v)\|_{}&\lesssim\|(\widetilde{\bm w}_h-\bm w,\widetilde{\bm v}_h-\bm v)\|_{}+\|(\curl(\bm w_h-\widetilde{\bm w}_h),\curl(\bm v_h-\widetilde{\bm v}_h))\|_{}\\
&\lesssim\|(\widetilde{\bm w}_h-\bm w,\widetilde{\bm v}_h-\bm v)\|_{\bm{\mathcal{H}}}+\|(\curl(\bm w_h-\bm w),\curl(\bm v_h-\bm v))\|_{}.
\end{align*}
The combination of the  above  two estimates and \eqref{proj}  gives the conclusion \eqref{pri}.
\end{proof}
\end{theorem}

The above  theorem  leads to the well-posedness of the problem \eqref{Problem_Weakdd} and  the stability
\begin{align}\label{stad}
\|(\bm w_h,\bm v_h)\|_{\bm{\mathcal H}}\lesssim \|(\bm f,\bm g)\|.
\end{align}

Next we shall denote $\mathcal{L}(X,Y)$ as the space of the linear functionals bounded from $X$ to $Y$ and simplify $\mathcal{L}(X,X)$ as $\mathcal{L}(X)$.
We introduce the operator $T\in \mathcal{L}(\mathbf L^2(\Omega),\bm{\mathcal{H}}_{})$ defined as
$$T(\bm f,\bm g):=(\bm w,\bm v)$$ from the problem \eqref{ProblemWeak} and the operator $T_h\in \mathcal{L}(\mathbf L^2(\Omega),\bm{\mathcal{H}}_{h})$ defined as
$$T_h(\bm f,\bm g):=(\bm w_h,\bm v_h)$$ from the problem \eqref{Problem_Weakdd}.
 Then the eigenvalue problem \eqref{Problem_Weak} has the operator form
\begin{eqnarray}
T(\bm w,\bm v)=\lambda^{-1} (\bm w,\bm v).
\end{eqnarray}

\begin{lemma}\label{Lemma 3.4}The operator $T$   is compact from  $\mathbf L^2(\Omega)\times \mathbf L^2(\Omega)$ to  $\mathbf L^2(\Omega)\times \mathbf L^2(\Omega)$.
\end{lemma}
\begin{proof}
Let  $\{(\bm f_i,\bm g_i)\}_{i=1}^\infty$  be  a bounded set in $\mathbf L^2(\Omega)\times \mathbf L^2(\Omega)$. Due to the stability condition \eqref{sta}, we see that $\{T(\bm f_i,\bm g_i)\}_{i=1}^\infty$ is a bounded set in $\bm{\mathcal{H}}_{0}$.
From \eqref{subset}, $\bm{\mathcal{H}}_{0}^S$ is compactly imbedded into $\mathbf L^2(\Omega)\times \mathbf L^2(\Omega)$.  Using the  decomposition \eqref{sum}, we have
$T(\bm f_i,\bm g_i)=(\bm{w}^S_i,\bm{v}^S_i)+(\nabla p^S_i,\nabla q^S_i)$  with $(\bm w^S_i,\bm v^S_i)\in \bm {\mathcal H}^S_{0}$ and $( p^S_i, q^S_i)\in \mathcal Q_{}$. 
    There is a subsequence of $\{(\bm{w}^S_i,\bm{v}^S_i)\}_{i=1}^\infty$ still denoted by itself such that
    $(\bm{w}^S_i,\bm{v}^S_i)$ converges strongly to some $\{(\bm{w}^S,\bm{v}^S)$ in $\mathbf L^2(\Omega)$ as $i\to \infty$.
But using the  decomposition \eqref{sum1}     we have  $(\bm{w}^S,\bm{v}^S)=  (\bm{w}^0,\bm{v}^0)+(\nabla p_{}^0,\nabla q_{}^0) $ for some $(\bm w^0,\bm v^0)\in \bm {\mathcal L}_{0}$ and $( p^0, q^0)\in \mathcal Q_{}$.  We now show that $T(\bm f_i,\bm g_i)$ converges to   $(\bm{w}^0,\bm{v}^0)$ in $\mathbf L^2(\Omega)$  as $i\to \infty$.  Using the fact that  $\left\{T(\bm f_i,\bm g_i)\right\}_{i=1}^{\infty} $ and  $(\boldsymbol{w}^0,\boldsymbol{v}^0)$  are in  $\bm{\mathcal{L}}_0$   we have
 $$c(T(\bm f_i,\bm g_i),(\nabla p',\nabla q')=c((\bm w^0,\bm v^0),(\nabla p'_{},\nabla q'_{})=0,~\forall ( p',  q')\in \mathcal Q_{}.$$
It follows that
       \begin{align}\label{cc1}
   &\quad c\left((\bm w^0,\bm v^0)-T(\bm f_i,\bm g_i),\mathbb T_c((\bm w^0,\bm v^0)-T(\bm f_i,\bm g_i))\right)\nonumber\\
&=c\left((\bm w^0,\bm v^0)-T(\bm f_i,\bm g_i),\mathbb T_c((\bm w^S,\bm v^S)-(\bm w_{i  }^S,\bm v_{i  }^S))\right).
    \end{align}
Hence,  $\left\|(\bm w^0,\bm v^0)-T(\bm f_i,\bm g_i)\right\|_{} \lesssim\left\|(\bm w^S-\bm w_{i  }^S,\bm v^S-\bm v_{i  }^S)\right\|_{ } \rightarrow 0$  as  $i\rightarrow 0 $.
\end{proof}

According to the spectral theory of linear compact operators, the  set of eigenvalues of the operator $T$  is  countable. Let $\lambda:=\lambda_m=\cdots=\lambda_{m+\beta-1}$ be the $m$th eigenvalue
of  $T$.  Let $\alpha$ and $\beta$ be the ascent and the  algebraic multiplicity of the eigenvalue $\lambda$ of $T$, respectively. Then $\alpha$ is the smallest positive integer  such that $\dim \Big({\ker}((T-\lambda^{-1})^{\alpha})\Big)=\beta$.

According to \cite{anselone,chatelin,osborn},  one says that the convergence $T_h\rightarrow T$ in $\mathcal L(\mathbf L^2(\Omega))\times \mathbf L^2(\Omega))$ is collectively compact if and only if

\textbf a) $T_h$ converges to $ T$ in $\mathbf L^2(\Omega)\times \mathbf L^2(\Omega)$  pointly, i.e., $\|(T_h-T)(\bm w',\bm v')\|\rightarrow 0,~\forall (\bm w',\bm v')\in \mathbf L^2(\Omega)\times \mathbf L^2(\Omega)$,

\textbf b) The set   $\cup_{h\in\mathcal M}( T- T_h)B$ is a relatively compact set in $\mathbf L^2(\Omega)\times \mathbf L^2(\Omega)$, where $B$ is the unit ball in $\mathbf L^2(\Omega)\times \mathbf L^2(\Omega)$.

%

\begin{theorem} There holds the collectively compact convergence
$T_h\rightarrow T$ in $\mathcal L(\mathbf L^2(\Omega))\times \mathbf L^2(\Omega))$.
\end{theorem}
\begin{proof}
In virtue of Lemma \ref{lemmadense}, $\cup_{h\in\mathcal M}\bm{\mathcal{H}}_h$ and $\cup_{h\in\mathcal M}  \mathcal Q_h$ are dense in $\bm{\mathcal{H}}$ and $\mathcal Q_h$, respectively. We deduce from \eqref{pri} that for any $(\bm f,\bm g)\in \mathbf L^2(\Omega))\times \mathbf L^2(\Omega)$
\begin{eqnarray*}
\|( T - T_h)(\bm f,\bm g)\|_{} \rightarrow 0.
 \end{eqnarray*}
 That is,  $ T_h$ converges to $ T$ pointwisely.
 From  the stability of $T$ and $T_h$ in  \eqref{sta} and  \eqref{stad}, $ T, T_h : \mathbf L^2(\Omega)\rightarrow \bm{\mathcal{H}}$ are linear bounded uniformly with respect to $h$, and $\cup_{h\in\mathcal M}( T- T_h)B$ is a bounded set in $\bm{\mathcal{H}}$. 
 From
  the compactness of $T$ in Lemma \ref{Lemma 3.4} and the discrete compactness property of $\bm{\mathcal{H}}_{0,h}$ in Lemma \ref{dcp}, we
know that $\cup_{h\in\mathcal M}( T- T_h)B$ is a relatively compact set in $\mathbf L^2(\Omega)\times \mathbf L^2(\Omega)$, which implies  collectively
compact convergence $ T_h\rightarrow T$.  
\end{proof}

Next we  define the  operator $T^*\in\mathcal L(\mathbf L^2(\Omega)\times \mathbf L^2(\Omega),\bm{\mathcal{H}}_0)$ as
\begin{align}\label{ad}
a((\bm \omega',\bm \upsilon'),T^*(\bm f,\bm g))=c((\bm \omega',\bm \upsilon'),(\bm f,\bm g)),~\forall (\bm \omega',\bm \upsilon')\in
 \bm{\mathcal{H}}_0.
\end{align}
Then $T^*$ is the adjoint operator of $T$ satisfying
\begin{align*}
c((\bm f,\bm g),T^*(\bm \omega',\bm \upsilon'))=c(T(\bm f,\bm g),(\bm \omega',\bm \upsilon')),~\forall (\bm f,\bm g),(\bm \omega',\bm \upsilon')\in
\mathbf L^2(\Omega)\times\mathbf L^2(\Omega).
\end{align*}
Similarly we can define the  operator $T_h^*\in\mathcal L(\mathbf L^2(\Omega)\times \mathbf L^2(\Omega),\bm{\mathcal{H}}_{0,h})$ as
\begin{align}\label{ad}
a((\bm \omega',\bm \upsilon'),T_h^*(\bm f,\bm g))=c((\bm \omega',\bm \upsilon'),(\bm f,\bm g)),~\forall (\bm \omega',\bm \upsilon')\in
 \bm{\mathcal{H}}_{0,h}.
\end{align}
Due to the definition of $T^*$ and $T_h^*$ and the symmetric properties  \eqref{sp} of $a(\cdot,\cdot)$ and $c(\cdot,\cdot)$, we see that $T^*=T$ and $T_h^*=T_h$. Hence both $\lambda$ and $\overline\lambda$ are   eigenvalues of $T$.

According to the  spectral approximation theory in \cite{babuska,osborn},
there are exactly $\beta$ eigenvalues $\lambda_{j,h}$ $({j=m,\cdots,m+\beta-1})$
of $T_h$ converging to $\lambda$ as $h\rightarrow0$. Now we are in a position to present the error estimate for finite element eigenpairs.
\begin{theorem}
Let $\lambda_{h  }\in \{\lambda_{j,h}\}_{j=m}^{m+\beta-1}$ be an eigenvalue of \eqref{Problem_Weakd} converging to  the eigenvalue $\lambda:=\lambda_m$ of \eqref{Problem_Weak} and $\ker((T-\lambda^{-1})^\alpha),\ker((T-\overline\lambda^{-1})^\alpha)\subset \mathbf H^s(\curl,\Omega)(s>1/2)$. Let $(\bm \omega_h,\bm \upsilon_h)$ be an eigenfunction corresponding to $\lambda_{h  }$ with $\|(\bm \omega_h,\bm \upsilon_{h})\|=1$ then there is    $(\bm \omega,\bm \upsilon)\in\ker(T-\lambda^{-1})$ such that
\begin{align}\label{4.19s}
|\lambda-{\lambda}_{h  }| &\lesssim h^{2\min(s,k+1)/\alpha}, \\
\|(\bm \omega_{h}-\bm \omega,\bm \upsilon_{h}-\bm \upsilon)\|_{\bm {\mathcal H}}&\lesssim h^{\min(s,k+1)/\alpha}.\label{l4}
\end{align}
\end{theorem}
\begin{proof} 
Note that any function $l\in\mathcal (\mathbf L^2(\Omega)\times \mathbf L^2(\Omega))'$ corresponds to a  linear bijective operator $\mathcal B\in\mathcal L((\mathbf L^2(\Omega)\times\mathbf L^2(\Omega))',\mathbf L^2(\Omega)\times\mathbf L^2(\Omega))$ such that $c(\mathcal B l,(\bm w',\bm v'))= l(\bm w',\bm v'),~\forall(\bm w',\bm v')\in \mathbf L^2(\Omega)\times \mathbf L^2(\Omega)$.
 From   Theorem 7.2  (inequality (7.12)), Theorem 7.3 and Theorem 7.4
in \cite{babuska} we obtain
\begin{align}
&|\lambda-{\lambda}_{h  }|^\alpha \lesssim \sum\limits_{i,j=1}^{\beta}| c((T-T_{h  })\bm\varphi_{i},\bm\varphi_{j}^*)|\nonumber\\
&\quad+\|(T-T_{h  })|_{\ker((T-\lambda^{-1})^\alpha)}  \|_{\mathcal L(\mathbf L^2(\Omega)\times \mathbf L^2(\Omega))}\|(T-T_{h  })|_{\ker((T-\overline\lambda^{-1})^\alpha)}  \|_{\mathcal L(\mathbf L^2(\Omega)\times \mathbf L^2(\Omega))},\label{l1}\\
&\|(\bm \omega_{h}-\bm \omega,\bm \upsilon_{h}-\bm \upsilon)\|^\alpha\lesssim \|(T-T_{h  })|_{\ker((T-\lambda^{-1})^\alpha)}  \|_{\mathcal L(\mathbf L^2(\Omega))\times \mathbf L^2(\Omega))},\label{l2}
\end{align}
where $\bm \varphi_{1},\cdots,\bm \varphi_{\beta}$ constitute  a set of  basis functions for $\ker((T-\lambda^{-1})^\alpha)$ and
$\bm \varphi_{1}^*,\cdots,\bm \varphi_{\beta}^*$ constitute  a set of  basis functions for $\ker((T-\overline\lambda^{-1})^\alpha)$.
It is  easy to see that $\ker((T-\lambda^{-1})^\alpha)$ and $\ker((T-\overline\lambda^{-1})^\alpha)$ are $T$-invariant  subspaces. According to   Theorem \ref{theorem3.1} and the interpolation estimate \eqref{i2} we further deduce from \eqref{l2}
\begin{align}\label{ll2}
\|(\bm \omega_{h}-\bm \omega,\bm \upsilon_{h}-\bm \upsilon)\|_{}\lesssim   h^{\min(s,k+1)/\alpha}.
\end{align}
Note that  $( p, q)= p_h, q_h)=0$ in (\ref{Problem_Weak}) and (\ref{Problem_Weakdd}) for $(\bm f,\bm g)\in  \ker((T-\lambda^{-1})^\alpha)\subset\bm{\mathcal{H}}_0$.  Moreover recalling the problems (\ref{Problem_Weak}) and (\ref{Problem_Weakdd}) and the symmetry properties \eqref{sp} of $a(\cdot,\cdot)$ and $c(\cdot,\cdot)$, we get
 \begin{align*}
  c((T-T_{h  })\bm\varphi_{i},\bm \varphi_{j})=  a((T-T_{h  })\bm\varphi_{i},T\bm \varphi_{j}^*)
  =  a((T-T_{h  })\bm\varphi_{i},(T-T_{h  })\bm \varphi_{j}^*).
 \end{align*}
Substituting this into \eqref{l1}, we deduce    (\ref{4.19s}) from Theorem \ref{theorem3.1} and \eqref{i2}.
By the boundedness \eqref{stad} of $T_{h}$  and   Theorem \ref{theorem3.1} we derive
\begin{align*}
&\quad\|(\bm \omega_{h}-\bm \omega,\bm \upsilon_{h}-\bm \upsilon)\|_{\bm{\mathcal{H}}}=\|\lambda_h(T_{h}(\bm \omega_{h},\bm \upsilon_{h})-T_{}(\bm \omega_{},\bm \upsilon_{}))\|_{\bm{\mathcal{H}}}\nonumber\\
&\leq \| \lambda_h T_{h}(\bm \omega_{h},\bm \upsilon_{h})-\lambda  T_{h}
(\bm \omega,\bm \upsilon)\|_{\bm{\mathcal{H}}}+\| \lambda T_{h} (\bm \omega_{},\bm \upsilon_{})- \lambda T
(\bm \omega_{},\bm \upsilon_{}) \|_{\bm{\mathcal{H}}}
\nonumber\\
&\lesssim \|\lambda_h(\bm \omega_{h},\bm \upsilon_{h})-\lambda
(\bm \omega_{},\bm \upsilon_{})\|_{}  +\| \lambda T_{h} (\bm \omega_{},\bm \upsilon_{})- \lambda T
(\bm \omega_{},\bm \upsilon_{}) \|_{\bm{\mathcal{H}}}\nonumber\\
&\lesssim |\lambda_h-\lambda|+
\|(\bm \omega_{h}-\bm \omega,\bm \upsilon_{h}-\bm \upsilon)\|_{} + \inf_{(\bm w',\bm v')\in \bm{\mathcal{H}}_h}\|T(\bm w,\bm v)-(\bm w',\bm v'))\|_{\bm{\mathcal{H}}}
\end{align*}
which together with \eqref{4.19s}, \eqref{ll2} and \eqref{i2}  yields the estimate \eqref{l4}.
\end{proof}

\begin{remark}The condition $\ker((T-\lambda^{-1})^\alpha)\subset \mathbf H^s(\curl,\Omega)(s>1/2)$ in the above theorem is reasonable. For the Maxwell's transmission problem \eqref{Problem_2ss} with $(\bm f,\bm g)\in \mathbf L^2(\Omega)\times \mathbf L^2(\Omega)$, we have $\curl(\bm A\curl\bm w),\curl^2 \bm v\in \mathbf L^2(\Omega)$ and $\bm \nu\times (A\curl\bm w-\curl \bm v)=\bm \nu\cdot (\curl\bm w-\curl \bm v)=0~{\rm on}\ \partial \Omega$. For simplicity, we assume that $\bm A=e(\bm x)I$ and $\bm N=n(\bm x)I$ where $ e(\bm x),n(\bm x)\in W^{1,\infty}(\Omega)$. Then $\mathrm{div}(\bm A\curl\bm w-\curl \bm v)\in \mathbf L^2(\Omega)$ and $\mathrm{div}(\curl\bm w-\curl \bm v)=0$.  According to Lemma \ref{lemmap1}, this implies that $\bm A\curl\bm w-\curl\bm v,\curl\bm w-\curl\bm v\in \mathbf H^r(\Omega)$ for some $r\in(1/2,1]$. In this case $\ker(T-\lambda^{-1})\subset \mathbf H^s(\curl,\Omega)(s\ge r)$. Furthermore, if $(\mathfrak{\bm w},\mathfrak{\bm v})\in\ker((T-\lambda^{-1})^{i+1})$ $(i=1,2,\cdots)$ then there is $(\bm f,\bm g)\in\ker((T-\lambda^{-1})^i)$ such that $T(\mathfrak{\bm w},\mathfrak{\bm v})=\lambda^{-1}(\mathfrak{\bm w},\mathfrak{\bm v})+(\bm f,\bm g)\in \mathbf L^2(\Omega)\times L^2(\Omega)$. Hence $\ker((T-\lambda^{-1})^\alpha)\subset \mathbf H^s(\curl,\Omega)(s\ge r)$. Similarly, the above argument is also true for $\ker((T-\overline\lambda^{-1})^\alpha)$.
\end{remark}

\section{Numerical verification}

\begin{table}
 \footnotesize
\caption{Numerical eigenvalues by linear edge element on the unit cube.}
\centering
\begin{tabular}{cccccccccccc}\hline
\multicolumn{9}{c}{$\bm A=2\bm I$,~$\bm N=16\bm I$}\\
$h$&$k_{1,h}$&$r_{1,h}$&$k_{2,h}$&$r_{2,h}$&$k_{3,h}$&$r_{3,h}$&$k_{4,h}$&$r_{4,h}$\\\hline
1/6&	1.20351& 	--& 	1.20369& 	--& 	1.20374& 	--& 	1.46189& 	--\\
1/7&	1.20512& 	1.89& 	1.20519& 	1.81& 	1.20524& 	1.82& 	1.46473& 	1.86\\
1/8&	1.20618& 	1.89& 	1.20625& 	1.91& 	1.20626& 	1.87& 	1.46673& 	1.99\\
1/9&	1.20694& 	1.96& 	1.20695& 	1.84& 	1.20697& 	1.86& 	1.46814& 	2.04\\
1/10&	1.20748& 	1.95& 	1.20749& 	1.93& 	1.20750& 	1.92& 	1.46907& 	1.90\\
1/11&	1.20787& 	1.88& 	1.20788& 	1.89& 	1.20789& 	1.87& 	1.46981& 	2.02
\\\hline
\multicolumn{9}{c}{$\bm A=F_1\bm I$,~$\bm N= F_2\bm I$}\\
$h$&$k_{1,h}$&$r_{1,h}$&$k_{2,h}$&$r_{2,h}$&$k_{3,h}$&$r_{3,h}$&$k_{4,h}$&$r_{4,h}$\\\hline
1/6&	4.39829& 	--& 	4.40194& 	--& 	4.40234& 	--& 	4.88009& 	--\\
1/7&	4.39598& 	2.29& 	4.39922& 	2.57& 	4.39995& 	2.24& 	4.88302& 	1.82\\
1/8&	4.39446& 	2.43& 	4.39776& 	2.25& 	4.39838& 	2.39& 	4.88502& 	1.87\\
1/9&	4.39352& 	2.32& 	4.39674& 	2.40& 	4.39732& 	2.46& 	4.88663& 	2.20\\
1/10&	4.39286& 	2.36& 	4.39610& 	2.18& 	4.39666& 	2.26& 	4.88765& 	1.96\\
1/11&	4.39239& 	2.33& 	4.39560& 	2.32& 	4.39617& 	2.32& 	4.88834& 	1.78
\\\hline
\multicolumn{9}{c}{$\bm A=\bm F_4$,~$\bm N=\bm F_3$}\\
$h$&$k_{1,h}$&$r_{1,h}$&$k_{2,h}$&$r_{2,h}$&$k_{3,h}$&$r_{3,h}$&$k_{4,h}$&$r_{4,h}$\\\hline
1/6&	3.85795& 	--& 	4.26145& 	--& 	4.43888& 	--& 	4.47280& 	--\\
1/7&	3.85954& 	1.66& 	4.26528& 	2.13& 	4.43646& 	2.32& 	4.47031& 	2.46\\
1/8&	3.86100& 	2.33& 	4.26733& 	1.74& 	4.43489& 	2.45& 	4.46890& 	2.27\\
1/9&	3.86187& 	2.09& 	4.26905& 	2.11& 	4.43386& 	2.47& 	4.46790& 	2.46\\
1/10&	3.86244& 	1.93& 	4.27014& 	1.87& 	4.43321& 	2.27& 	4.46728& 	2.23\\
1/11&	3.86288& 	1.95& 	4.27096& 	1.87& 	4.43275& 	2.28& 	4.46680& 	2.36
\\\hline
\end{tabular}
\end{table}

\begin{table}
 \footnotesize
\caption{Numerical eigenvalues by linear edge element on the thick-L domain.}
\centering
\begin{tabular}{cccccccccccc}\hline
\multicolumn{9}{c}{$\bm A=2\bm I$,~$\bm N=16\bm I$}\\
$h$&$k_{1,h}$&$r_{1,h}$&$k_{2,h}$&$r_{2,h}$&$k_{3,h}$&$r_{3,h}$&$k_{4,h}$&$r_{4,h}$\\\hline
1/4&	0.82137& 	--& 	0.89727& 	--& 	0.99069 &	--& 	1.07155 &	--\\
1/5&	0.82537& 	1.67& 	0.90022& 	2.05& 	0.99314 &	1.96 &	1.07400 &	1.95\\
1/6&	0.82776& 	1.73& 	0.90224& 	2.78 &	0.99450 &	1.99 &	1.07535 &	1.94\\
1/7&	0.82928& 	1.73& 	0.90252& 	0.64 &	0.99535 &	2.10 &	1.07619 &	2.00\\
1/8&	0.83039& 	1.90& 	0.90331& 	2.51 &	0.99588 &	2.02 &	1.07669 &	1.84\\
1/9&	0.83115& 	1.89& 	0.90377& 	2.26 &	0.99624 &	1.97 &	1.07708& 	2.06
\\\hline
\multicolumn{9}{c}{$\bm A=F_1\bm I$,~$\bm N= F_2\bm I$}\\
$h$&$k_{1,h}$&$r_{1,h}$&$k_{2,h}$&$r_{2,h}$&$k_{3,h}$&$r_{3,h}$&$k_{4,h}$&$r_{4,h}$\\\hline
1/4&	3.12372 &	-- &	3.12568 &	--& 	3.27184 &	--& 	3.36932 &	--\\
1/5&	3.11590 &	1.47 &	3.13604 &	1.41 &	3.28111 &	2.29 &	3.36829 &	--\\
1/6&	3.11026 &	1.81 &	3.14370 &	1.75 &	3.28676 &	2.86 &	3.36807 &	--\\
1/7&	3.10674 &	1.81 &	3.14850 &	1.74 &	3.28727 &	0.41 &	3.36840 &	0.82\\
1/8&	3.10425 &	1.93 &	3.15205 &	1.94 &	3.28956 &	2.62 &	3.36868 &	0.89\\
1/9&	3.10248& 	2.00 &	3.15441& 	1.86 &	3.29082& 	2.23 &	3.36895& 	1.11
\\\hline
\multicolumn{9}{c}{$\bm A=\bm F_4$,~$\bm N=\bm F_3$}\\
$h$&$k_{1,h}$&$r_{1,h}$&$k_{2,h}$&$r_{2,h}$&$k_{3,h}$&$r_{3,h}$&\multicolumn{2}{c}{$k_{4,h}$}&$r_{4,h}$\\\hline
1/4&	2.67101& 	--& 	3.13237 &	--& 	3.27326 &	--& 	\multicolumn{2}{c}{3.38505+0.028300i}&	--\\
1/5&	2.68064& 	1.55& 	3.12237 &	1.72 &	3.27839 &	2.08 &	\multicolumn{2}{c}{3.38540+0.027958i}&	0.69\\
1/6&	2.68656& 	1.60& 	3.11666 &	1.71 &	3.28251 &	3.53 &	\multicolumn{2}{c}{3.38563+0.027840i}&	0.63\\
1/7&	2.69035& 	1.59& 	3.11311 &	1.67 &	3.28238 &	-0.17 &	\multicolumn{2}{c}{3.38586+0.027832i}&	0.86\\
1/8&	2.69318& 	1.74& 	3.11059& 	1.75 &	3.28388 &	2.88 &\multicolumn{2}{c}{3.38602+0.027833i}&	0.72\\
1/9&	2.69508& 	1.64& 	3.10880& 	1.76 &	3.28464& 	2.30 &	\multicolumn{2}{c}{3.38621+0.027831i}&	1.19
\\\hline
\end{tabular}
\end{table}

\begin{table}
 \footnotesize
\caption{Numerical eigenvalues by quadratic edge element on the unit cube.}
\centering
\begin{tabular}{cccccccccccc}\hline
\multicolumn{9}{c}{$\bm A=2\bm I$,~$\bm N=16\bm I$}\\
$h$&$k_{1,h}$&$r_{1,h}$&$k_{2,h}$&$r_{2,h}$&$k_{3,h}$&$r_{3,h}$&$k_{4,h}$&$r_{4,h}$\\\hline
1/3&	1.209189 &	--& 	1.209218 &	--& 	1.209314 &	--& 	1.471039 &	--\\
1/4&	1.209586 &	3.03 &	1.209593 &	2.97 &	1.209595 &	2.45 &	1.472310 &	2.89\\
1/5&	1.209744 &	3.62 &	1.209745 &	3.56 &	1.209749 &	3.67 &	1.472861 &	3.69\\
1/6&	1.209803 &	3.50 &	1.209805 &	3.59 &	1.209807 &	3.54 &	1.473065 &	3.50\\
1/7&	1.209834 &	3.98 &	1.209835 &	3.93 &	1.209835 &	3.89 &	1.473167 &	3.88\\
1/8&	1.209850 &	4.31 &	1.209850 &	4.30 &	1.209851 &	4.36 &	1.473220 &	4.09\\
Ref.&	1.209871& 	-- &	1.209870& 	--&	1.209870& 	-- &	1.473293 &	--
\\\hline
\multicolumn{9}{c}{$\bm A=F_1\bm I$,~$\bm N= F_2\bm I$}\\
$h$&$k_{1,h}$&$r_{1,h}$&$k_{2,h}$&$r_{2,h}$&$k_{3,h}$&$r_{3,h}$&$k_{4,h}$&$r_{4,h}$\\\hline
1/3&	4.391999 &	--& 	4.395247 &	-- &	4.395818 &	-- &	4.893018 &	--\\
1/4&	4.391001 &	3.87 &	4.394126 &	4.02 &	4.394706 &	3.96 &	4.892394 &	3.77\\
1/5&	4.390704 &	4.20 &	4.393812 &	4.23 &	4.394379 &	4.39 &	4.892225 &	3.41\\
1/6&	4.390606 &	3.96 &	4.393707 &	4.11 &	4.394278 &	3.93 &	4.892147 &	4.03\\
1/7&	4.390562 &	4.17 &	4.393664 &	3.85 &	4.394233 &	4.16 &	4.892115 &	3.91\\
1/8&	4.390542 &	3.88 &	4.393643 &	4.06 &	4.394212& 	3.91 &	4.892098 &	4.06\\
Ref.&	4.390513& 	-- &	4.393612& 	-- &	4.394182 &	-- &	4.892076& 	--
\\\hline
\multicolumn{9}{c}{$\bm A=\bm F_4$,~$\bm N=\bm F_3$}\\
$h$&$k_{1,h}$&$r_{1,h}$&$k_{2,h}$&$r_{2,h}$&$k_{3,h}$&$r_{3,h}$&$k_{4,h}$&$r_{4,h}$\\\hline
1/3&	3.866098& 	--& 	4.275852 &	--& 	4.432275 &	--& 	4.466491 &	--\\
1/4&	3.865268& 	4.76& 	4.275364 &	4.17 &	4.431310 &	3.80 &	4.465404 &	4.11\\
1/5&	3.865103& 	3.92& 	4.275249 &	3.56 &	4.431012 &	4.26 &	4.465107 &	4.34\\
1/6&	3.865043& 	3.97 &	4.275201 &	3.86 &	4.430912 &	4.14 &	4.465010 &	4.12\\
1/7&	3.865018& 	3.75& 	4.275177 &	4.64 &	4.430870 &	4.13 &	4.464972 &	3.91\\
1/8&	3.865004& 	4.22 &	4.275168 &	3.55 &	4.430852 &	3.80 &	4.464952 &	4.06\\
Ref.&	3.864985& 	-- &	4.275154& 	-- &	4.430824& 	-- &	4.464924& 	--
\\\hline
\end{tabular}
\end{table}

\begin{table}
 \footnotesize
\caption{Numerical eigenvalues by quadratic edge element on the thick L-shaped domain.}
\centering
\begin{tabular}{cccccccccccc}\hline
&\multicolumn{4}{c}{$\bm A=2\bm I$,~$\bm N=16\bm I$}&\multicolumn{4}{c}{$\bm A=F_1\bm I$,~$\bm N= F_2\bm I$}\\
$h$&$k_{1,h}$&$k_{2,h}$&$k_{3,h}$&$k_{4,h}$&$k_{1,h}$&$k_{2,h}$&$k_{3,h}$&$k_{4,h}$\\\hline
1&0.81995& 	0.89721&	0.99263&	1.07513&3.12559 &	3.15480 & 	3.27532&  	3.37076 \\
1/2&0.83002 &	0.90371 &	0.99656 &	1.07640&3.10592 &	3.15133 & 	3.29075 & 	3.36773 \\
1/3&0.83250 &	0.90477 &	0.99731 &	1.07792&3.09997 &	3.15836 & 	3.29324 & 	3.36915 \\
1/4&0.83346 &	0.90506 &	0.99749 &	1.07830&3.09773 &	3.16144 & 	3.29411 & 	3.37002 \\
1/5&0.83393 &	0.90519 &	0.99756 &	1.07844&3.09653 &	3.16301 & 	3.29463 & 	3.37054 \\
1/6&0.83423 & 	0.90529 & 	0.99759 &	1.07851&3.09579& 	3.16405 & 	3.29502 & 	3.37087
\\\hline
\multicolumn{6}{c}{$\bm A=\bm F_4$,~$\bm N=\bm F_3$}\\
$h$&$k_{1,h}$&$k_{2,h}$&$k_{3,h}$&$k_{4,h}$\\\hline
1&	2.67570& 	3.16229& 	3.28325& 	\multicolumn{2}{c}{3.39955+0.026382i}\\
1/2&	2.69202& 	3.11227& 	3.28659& 	\multicolumn{2}{c}{3.38610+0.027600i}\\
1/3&	2.69799& 	3.10626& 	3.28673& 	\multicolumn{2}{c}{3.38644+0.027763i}\\
1/4&	2.70061& 	3.10400& 	3.28683& 	\multicolumn{2}{c}{3.38685+0.027848i}\\
1/5&	2.70195& 	3.10279& 	3.28684& 	\multicolumn{2}{c}{3.38712+0.027884i}\\
1/6&	2.70282& 	3.10204& 	3.28695& 	\multicolumn{2}{c}{3.38729+0.027908i}
\\\hline
\end{tabular}
\end{table}

In this section, let $\Omega=(0,1)^3$  or the thick L-shape domain $\Omega=((-1,1)^2\backslash(-1,0]^2)\times(0,1)$.
For the setting of $\bm A$ and $\bm N$, we consider the following three cases\\

(1) $\bm A=2\bm I$, $\bm N=16\bm I$, (2) $\bm A=F_1\bm I$, $\bm N=F_2\bm I$ and (3) $\bm A=\bm F_4 $, $\bm N=\bm F_3$,
\\
where
\begin{eqnarray*}
& F_1=\exp(x_1+x_2+x_3)+6,~ F_2=8+x_1-x_2+x_3,~\bm F_3=\left[
\begin{tabular}{ccc}
16& $x_1$& $x_2$\\$x_1$& 16& $x_3$\\$x_2$& $x_3$& 14
\end{tabular}
\right],\\
&\bm F_4
=\left[
\begin{tabular}{ccc}
$- \frac{x_1^2}8 + \frac{9x_1}8 - \frac{9x_2}8 + \frac{65}4$& $\frac{3x_1}8 - \frac{3x_2}8 - \frac{3x_1^2}8 + \frac3 4$&        0\\
 $\frac{3x_1^2}8 - \frac{3x_1}8 + \frac{3x_2}8 - \frac3 4$&       $ \frac{9x_1^2}8 - \frac{x_1}8 + \frac{x_2}8 + \frac{55}4$&        0\\
                                   0&                                   0& $x_3^2 + 12$
\end{tabular}
\right].
\end{eqnarray*}
It is known from \cite{sun2} that  the lowest eigenvalue of $\bm F_3$ exceeds 1.
Moreover, the   eigenvalues of $\bm F_4$   are    $x_1^2 + 14$,   $x_3^2 + 12$ and $x_1 - x_2 + 16$, minimum of which  exceeds 1. These settings of $\bm A$ and $\bm N$ ensure that $A_*>1$ and $N_*>1$.

In the mesh refinement process, the unit cube or  the thick L-shape domain is partitioned into
small congruent cubes, and each small cube is further divided into 12 tetrahedra. 
The mesh data structure utilized in this work is provided by the iFEM package within MATLAB \cite{chenlong}.

The computed lowest four eigenvalues obtained by the linear and quadratic edge element methods
are provided in Tables 1-4. When the mesh size $h$ on the cube (resp. the thick L-shape domain) ranges   $1/6\sim1/11$ (resp. $1/4\sim1/9$ ) the number of degrees of freedom in  the linear  edge element method ranges   $6084\sim37334$ (resp. $5416\sim61326$ ). When the mesh size $h$ on the cube (resp. the thick L-shape domain) ranges   $1/3\sim1/8$ (resp. $1\sim1/6$ ) the number of degrees of freedom in  the quadratic  edge element method ranges   $4140\sim77920$ (resp. $476\sim98616$ ).

We use the following formula to derive the reference values of the eigenvalue $k_j$ on the cube:
$k_j\approx k_{j,h_{end}}-Mean_i (C_{j,i}) \times h_{end}^4$ where  $k_{j,h_i}-k_{j,h_{i+1}}=C_{j,i}(h_i^4-h_{i+1}^4)$ and $h_{end}$ is
  the smallest  mesh size  available from computation. With this formula and the numerical eigenvalues obtained by the quadratic edge element method, the reference eigenvalues  on the cube  are listed in Table 3. Then we can utilize the formula
\[r_{j,h_{i}}=\frac{\log(|k_j-k_{j,h_i}|)-\log(|k_{j}-k_{j,h_{i-1}}|))}{\log(h_{i})-\log(h_{i-1})}\]
to approximately compute the convergence rate $r_{j,h_{i}}$ of the numerical eigenvalue $k_{j,h_i}$.

Obtaining the reference eigenvalue on the thick L-shape domain poses some challenges due to the concave dihedral angle of the domain, which may result in singularity of the associated eigenfunction.
Therefore we regard the numerical eigenvalue  obtained by the quadratic edge element method on a fine mesh  as
the 'exact' eigenvalue for computing the convergence rate  of the numerical eigenvalue obtained by the linear edge element method.
  The convergence rates of the numerical eigenvalues obtained by the linear edge element method (resp. the quadratic edge element method) are presented in Tables 1 and 2 (resp. Table 3). We can see from Tables 1 and 2 that the convergence rates of the numerical eigenvalues on the cube obtained by the linear and quadratic edge element methods are approximately two and four, respectively,  consistent with our theoretical expectations.   Table 2 indicates that the convergence rates of the numerical eigenvalues on the thick L-shape domain obtained by the linear edge element  methods do not exceed two; in particular, when $\bm A =  F_1\bm I$ and $\bm N =  F_2\bm I$ or $\bm A = \bm F_4$ and $\bm N = \bm F_3$ the fourth  numerical eigenvalue on this domain converges  showly.

Let $(\bm \omega_1,\bm \upsilon_1)$ denote  the first eigenfunction. We present the computed  $\bm \upsilon_1$ on the cube in Fig. 1 and  $\bm\omega_1-\bm \upsilon_1$ in Fig. 2. Due to the limited length of this paper, we only display  the computed  $\bm\omega_1-\bm \upsilon_1$ on the thick L-shape domain in Fig. 3.  The numerical results for the remaining cases   2)$\sim $4)  of $\bm A$ and $\bm N$ are similar and therefore not listed here.

\begin{figure}[htp]
\caption{The three components of   $\bm \upsilon_1$ on the cube with $\bm A=2\bm I,\bm N = 16\bm I$ (at the top), $\bm A=F_1\bm I,\bm N = F_2\bm I$ (at the middle) and $\bm A=\bm F_4,\bm N = \bm F_3$ (at the bottom) computed by the quadratic edge element method.}
\label{fig-square}
\begin{center}
\includegraphics[width=4.2cm, height=4.2cm]{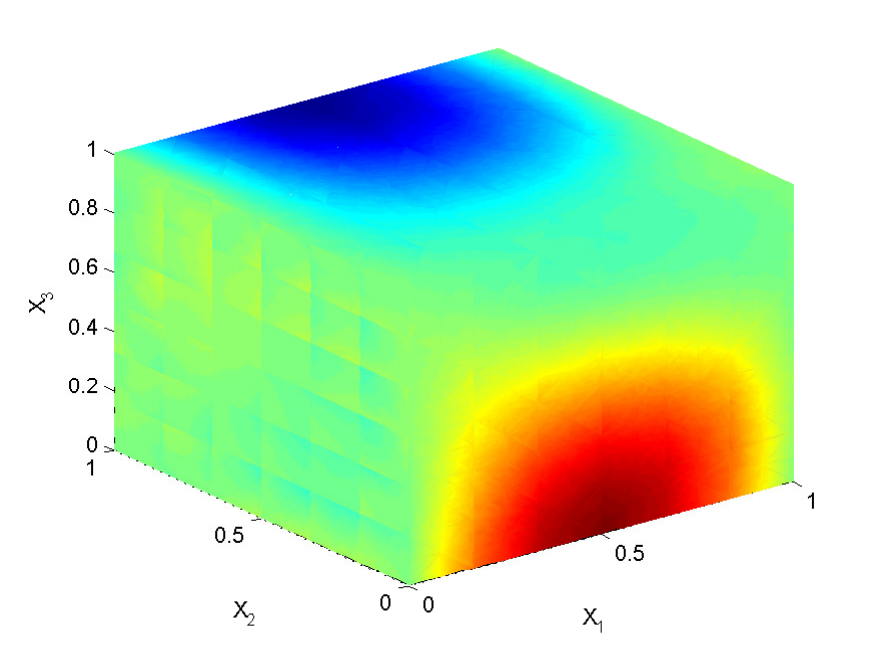}
\includegraphics[width=4.2cm, height=4.2cm]{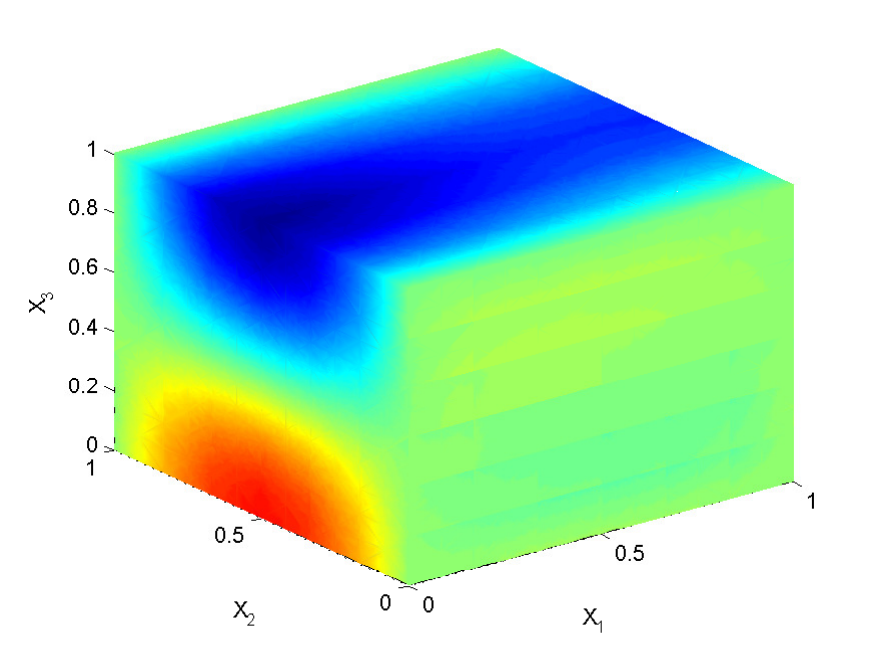}
\includegraphics[width=4.2cm, height=4.2cm]{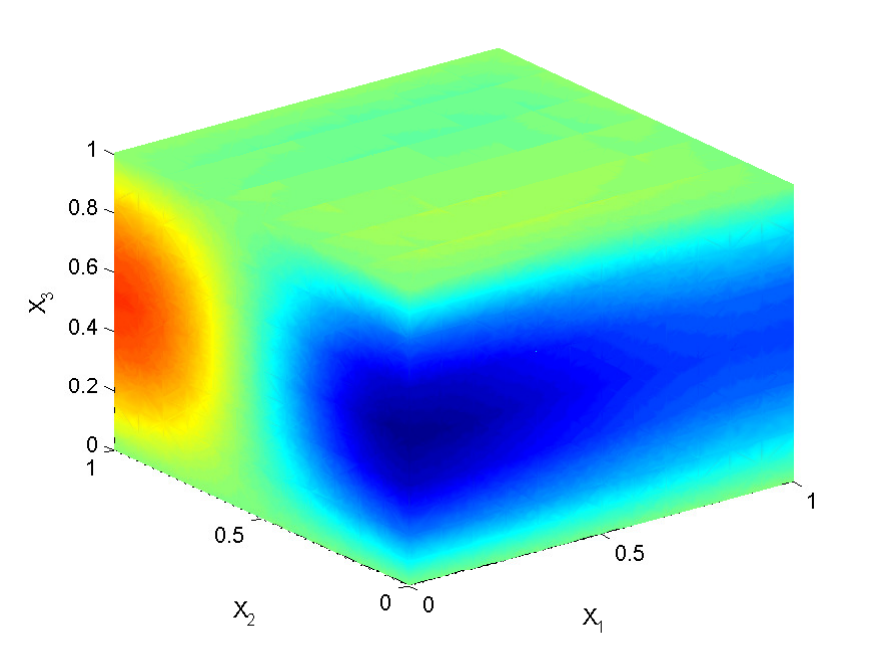}\\
\includegraphics[width=4.2cm, height=4.2cm]{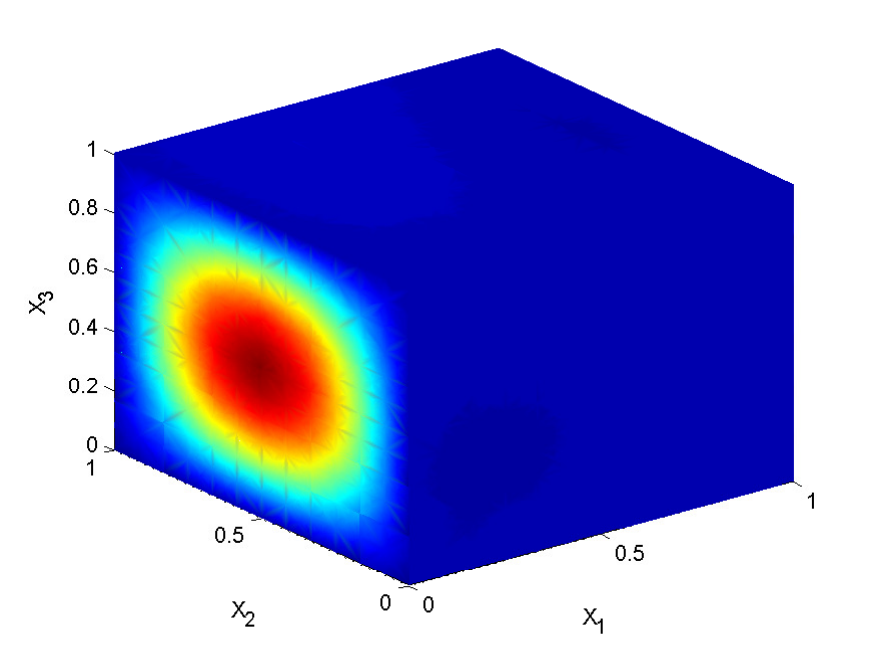}
\includegraphics[width=4.2cm, height=4.2cm]{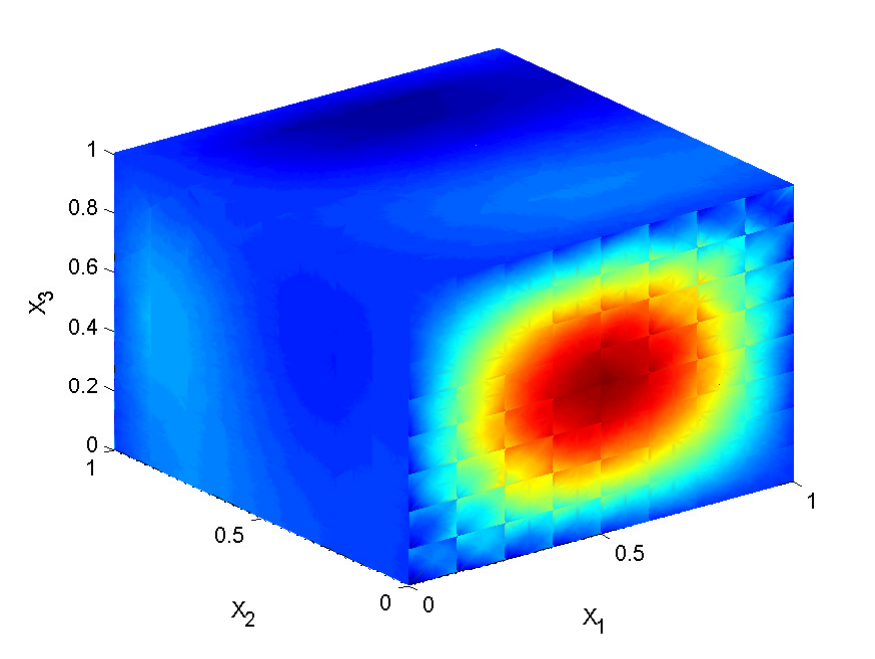}
\includegraphics[width=4.2cm, height=4.2cm]{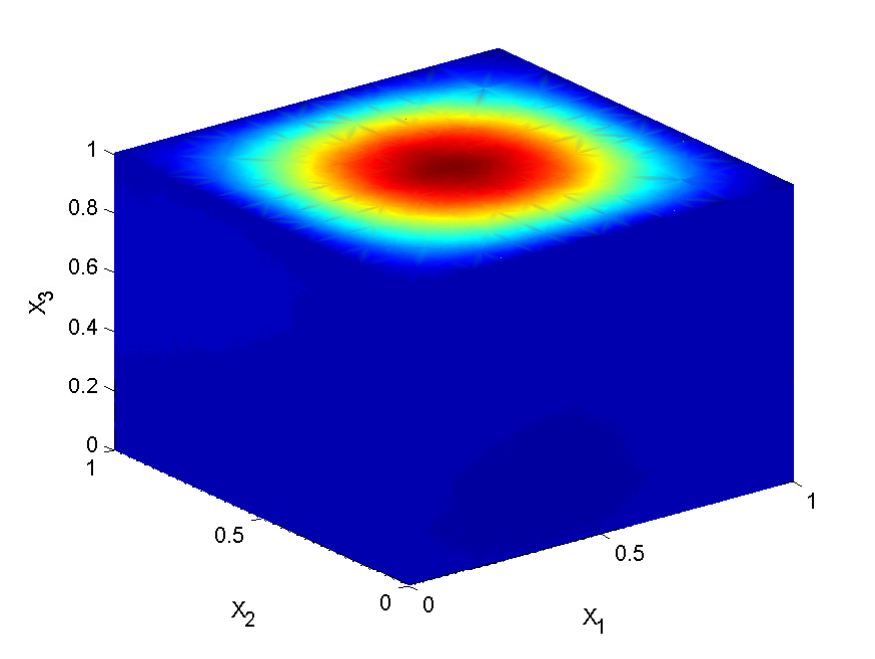}\\
\includegraphics[width=4.2cm, height=4.2cm]{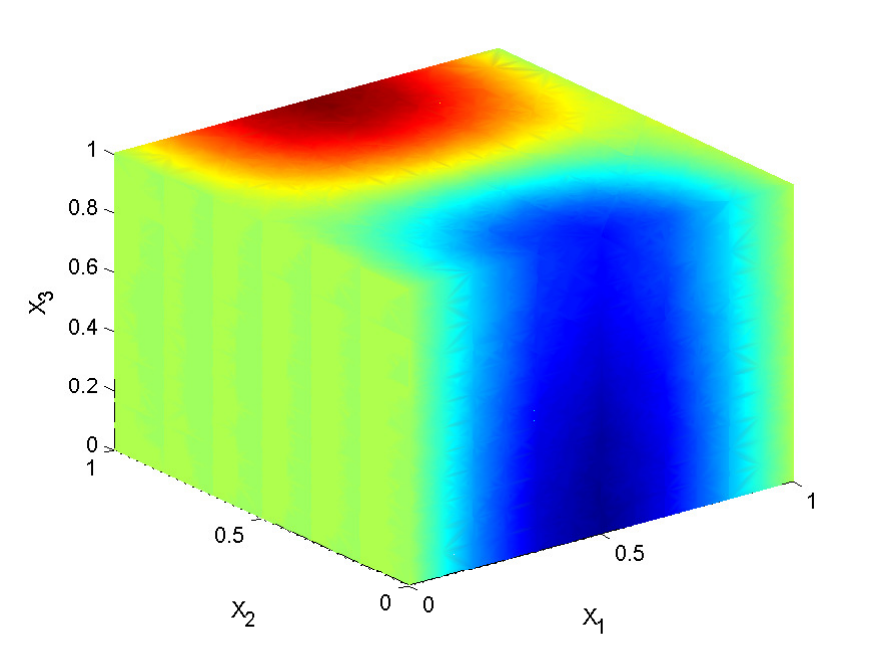}
\includegraphics[width=4.2cm, height=4.2cm]{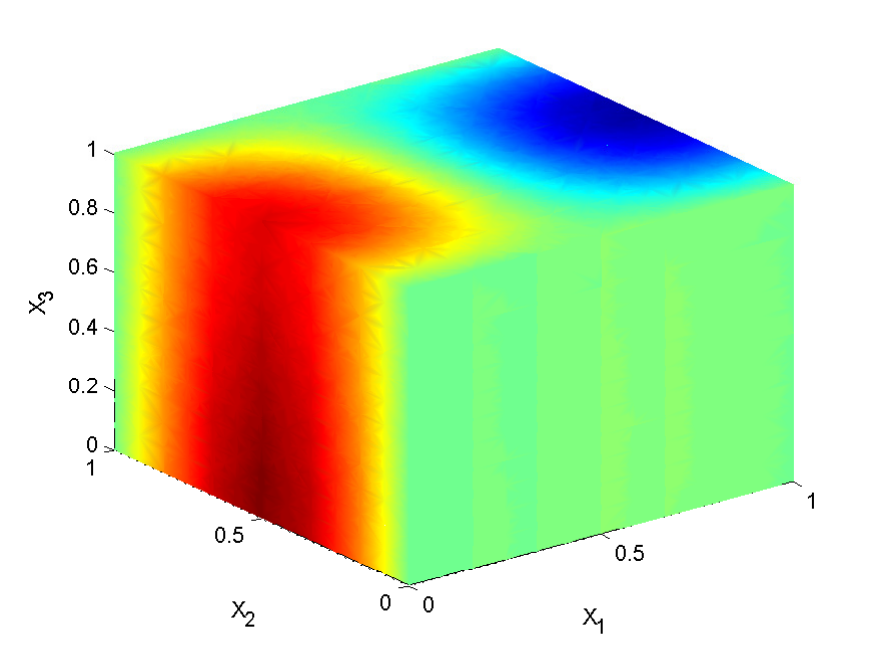}
\includegraphics[width=4.2cm, height=4.2cm]{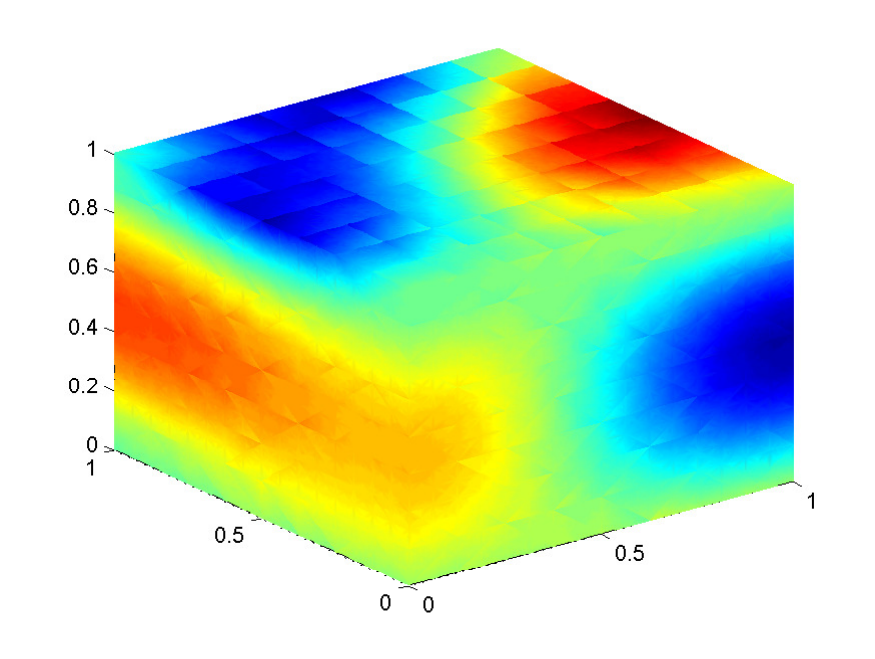}
\end{center}
\end{figure}

\begin{figure}[htp]
\caption{The three components of   $\bm \omega_1-\bm \upsilon_1$ on the cube with $\bm A=2\bm I,\bm N = 16\bm I$ (at the top), $\bm A=F_1\bm I,\bm N = F_2\bm I$ (at the middle) and $\bm A=\bm F_4,\bm N = \bm F_3$ (at the bottom) computed by the quadratic edge element method.}
\label{fig-square}
\begin{center}
\includegraphics[width=4.2cm, height=4.2cm]{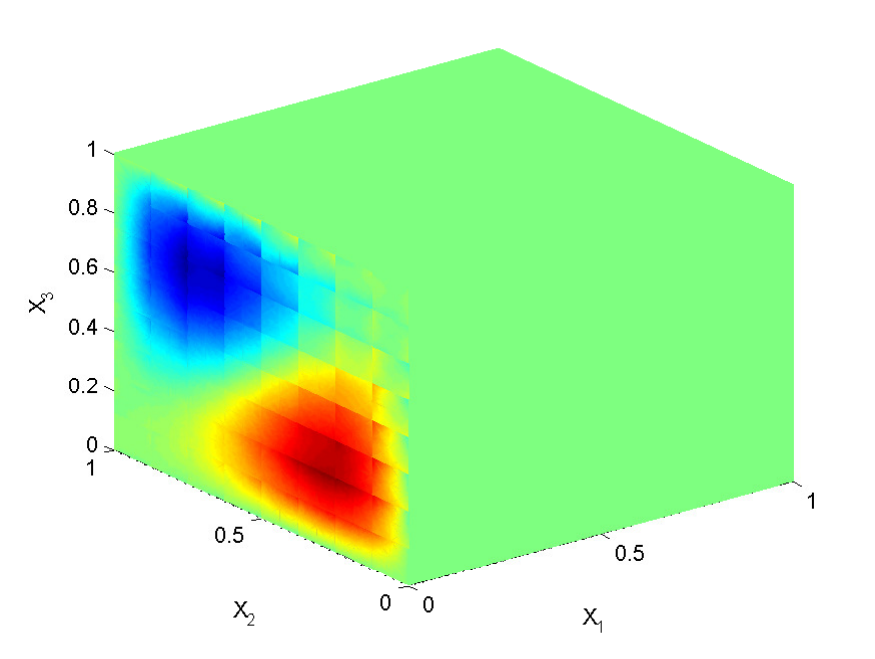}
\includegraphics[width=4.2cm, height=4.2cm]{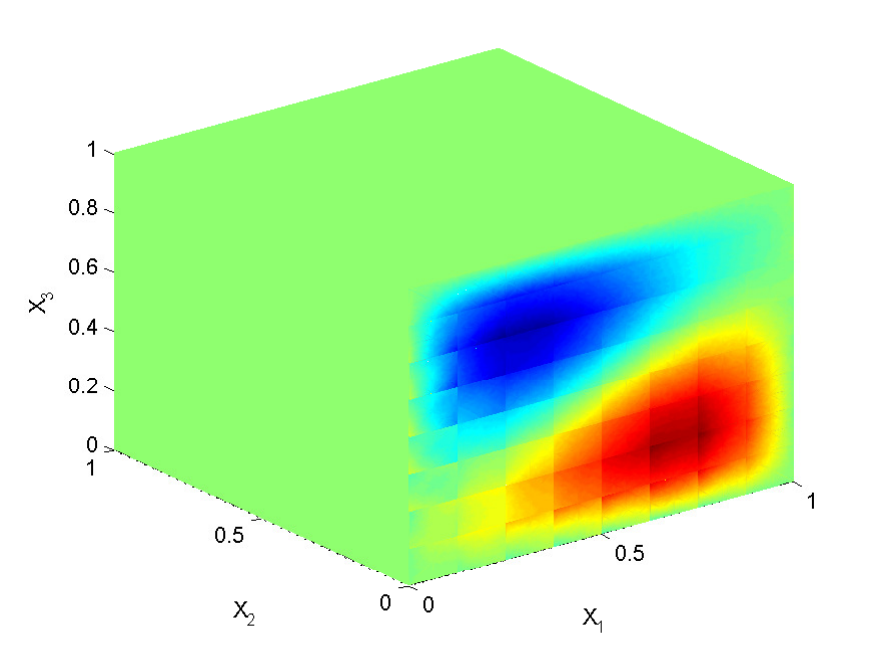}
\includegraphics[width=4.2cm, height=4.2cm]{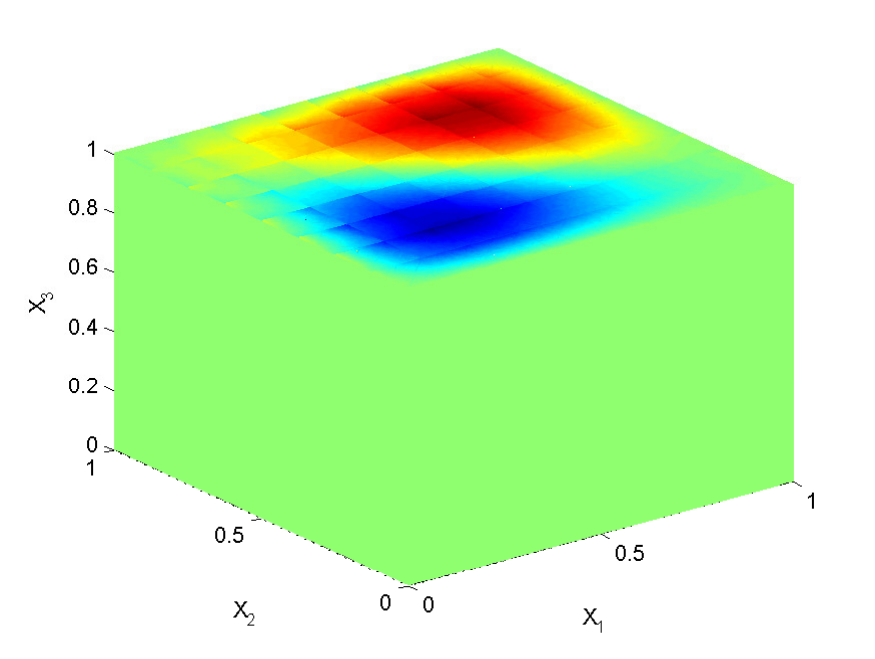}\\
\includegraphics[width=4.2cm, height=4.2cm]{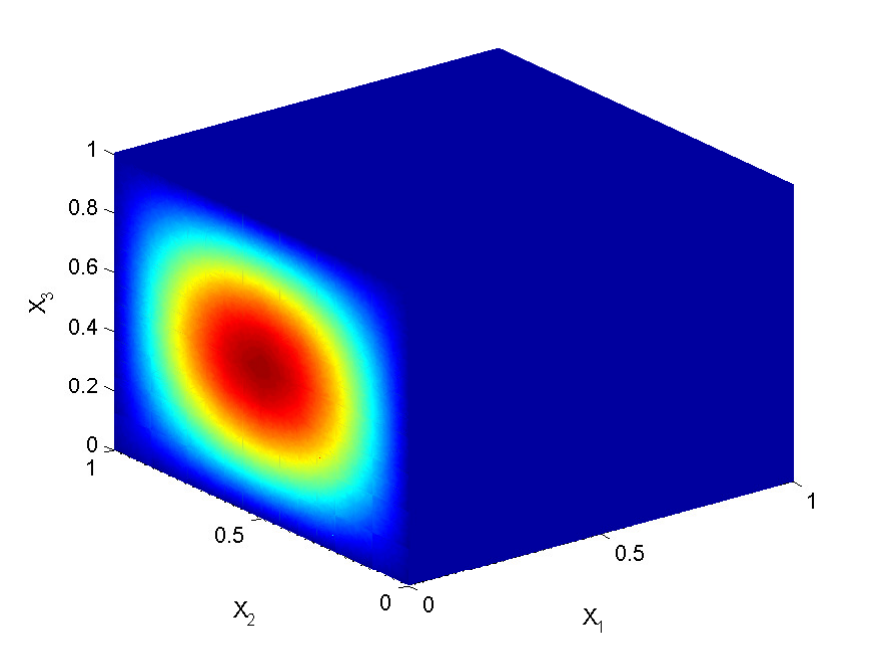}
\includegraphics[width=4.2cm, height=4.2cm]{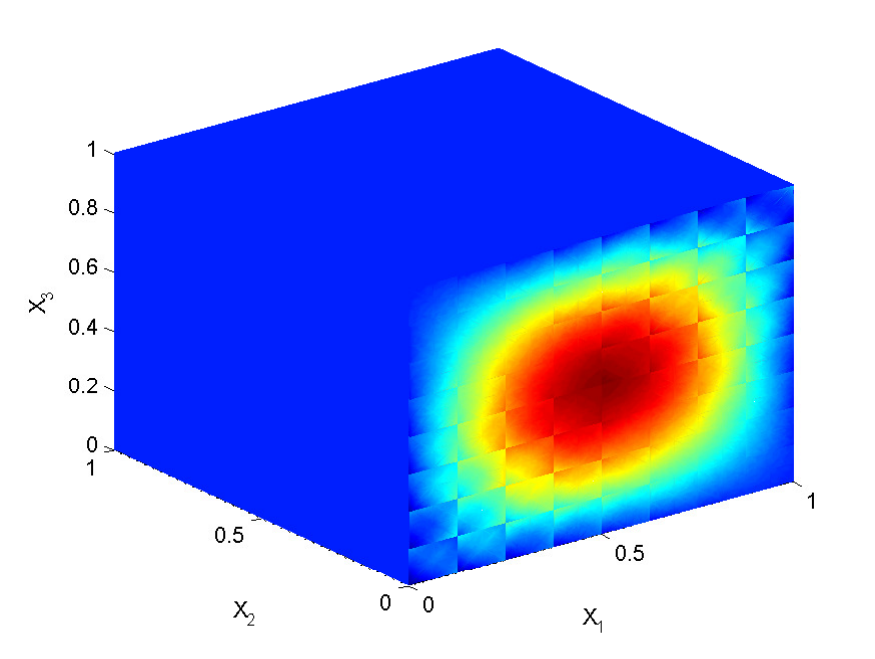}
\includegraphics[width=4.2cm, height=4.2cm]{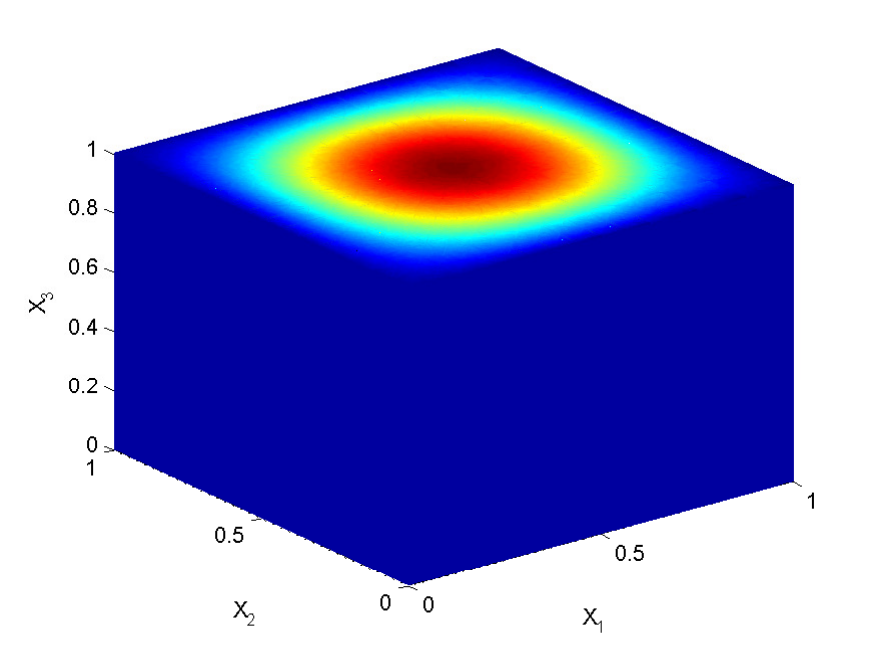}\\
\includegraphics[width=4.2cm, height=4.2cm]{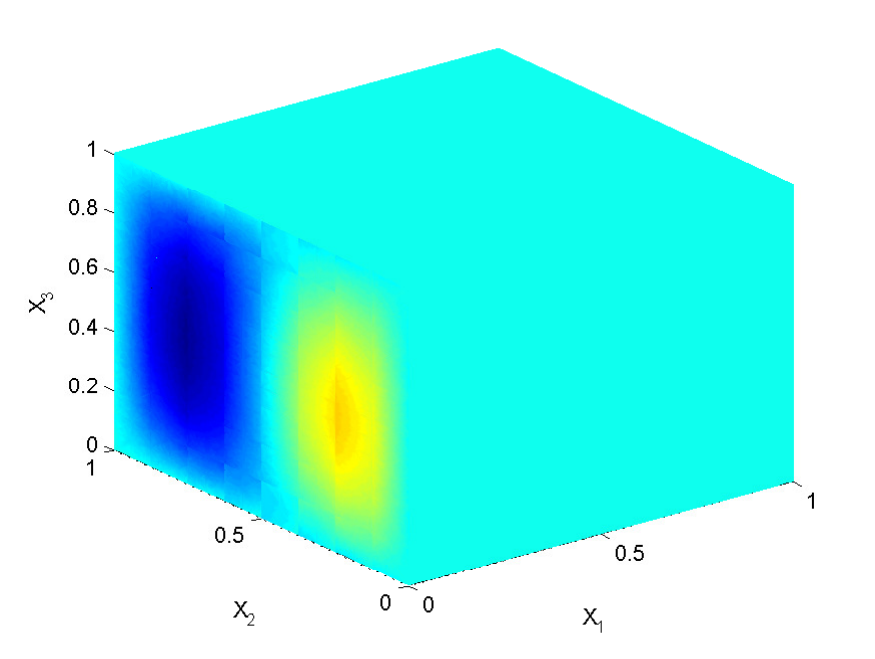}
\includegraphics[width=4.2cm, height=4.2cm]{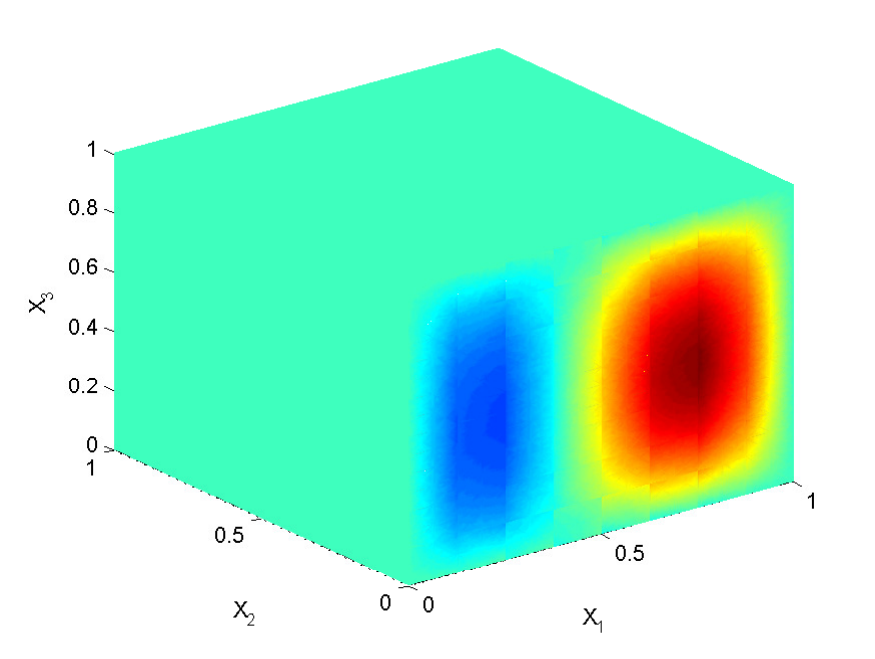}
\includegraphics[width=4.2cm, height=4.2cm]{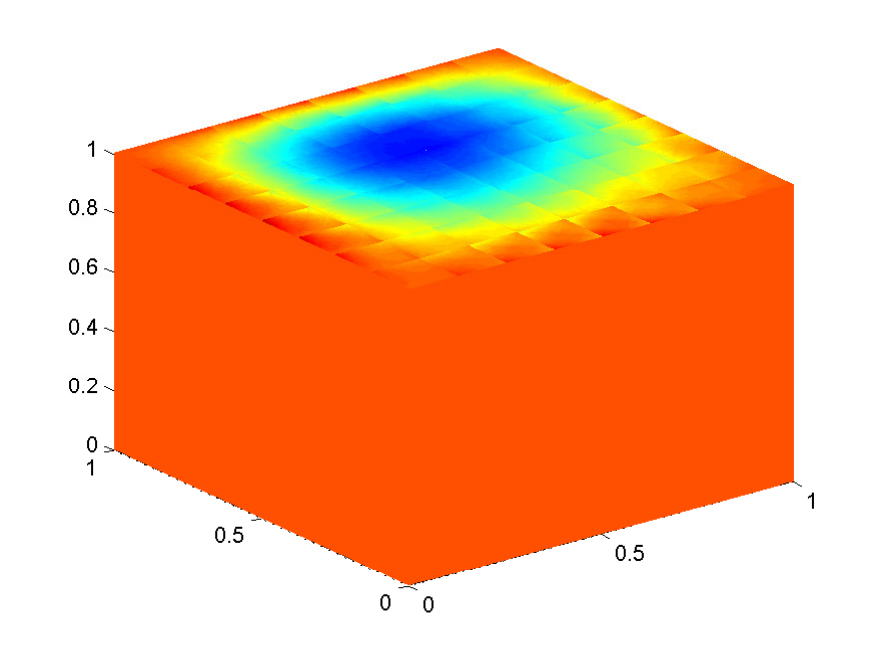}
\end{center}
\end{figure}

\begin{figure}[htp]
\caption{The three components of $\bm\omega_1-\bm \upsilon_1$ on the thick L-shape domain with $\bm A=2\bm I,\bm N = 16\bm I$ (at the top), $\bm A=F_1\bm I,\bm N = F_2\bm I$ (at the middle) and $\bm A=\bm F_4,\bm N = \bm F_3$ (at the bottom) computed by the quadratic edge element method.}
\label{fig-square}
\begin{center}
\includegraphics[width=4.2cm, height=4.2cm]{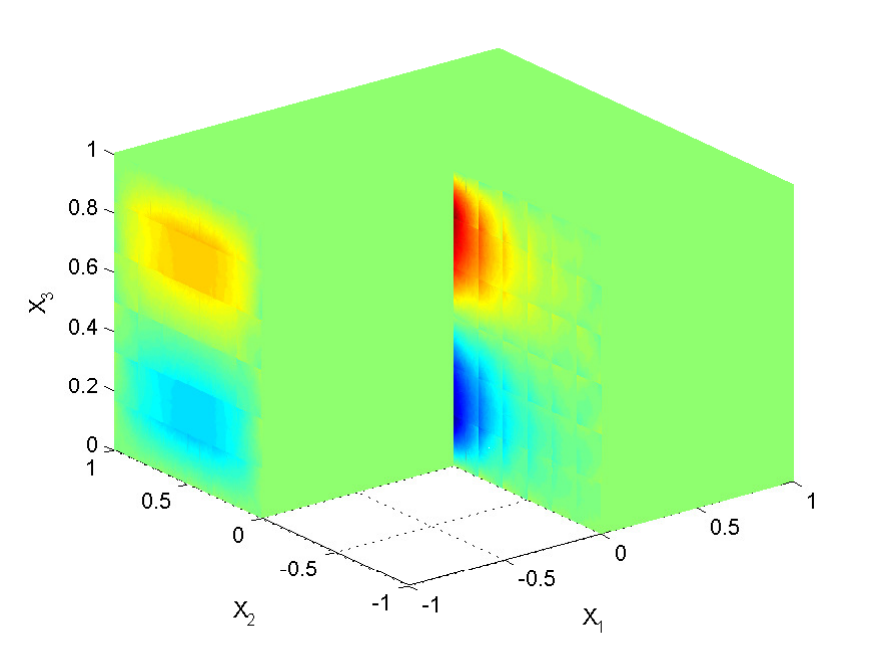}
\includegraphics[width=4.2cm, height=4.2cm]{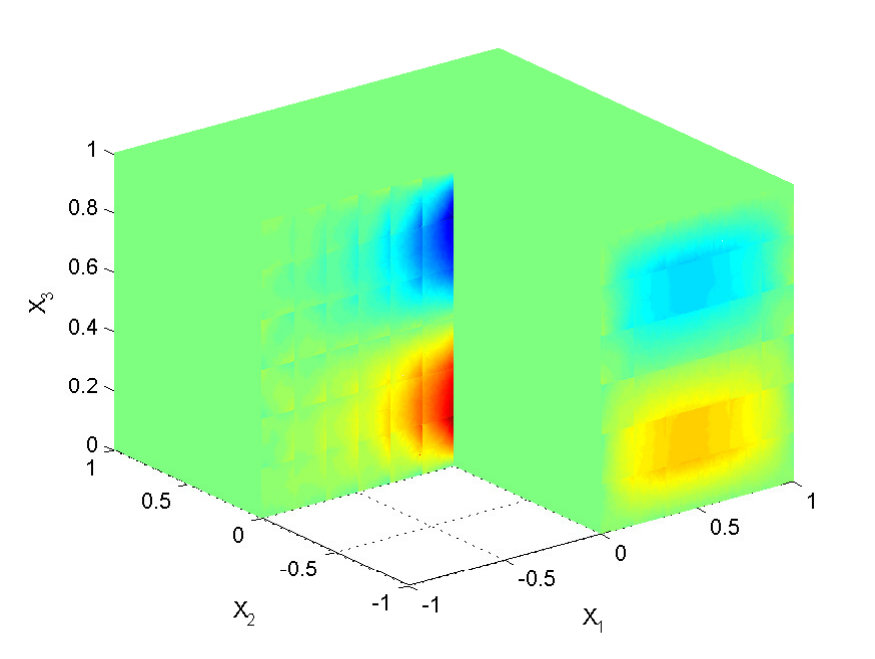}
\includegraphics[width=4.2cm, height=4.2cm]{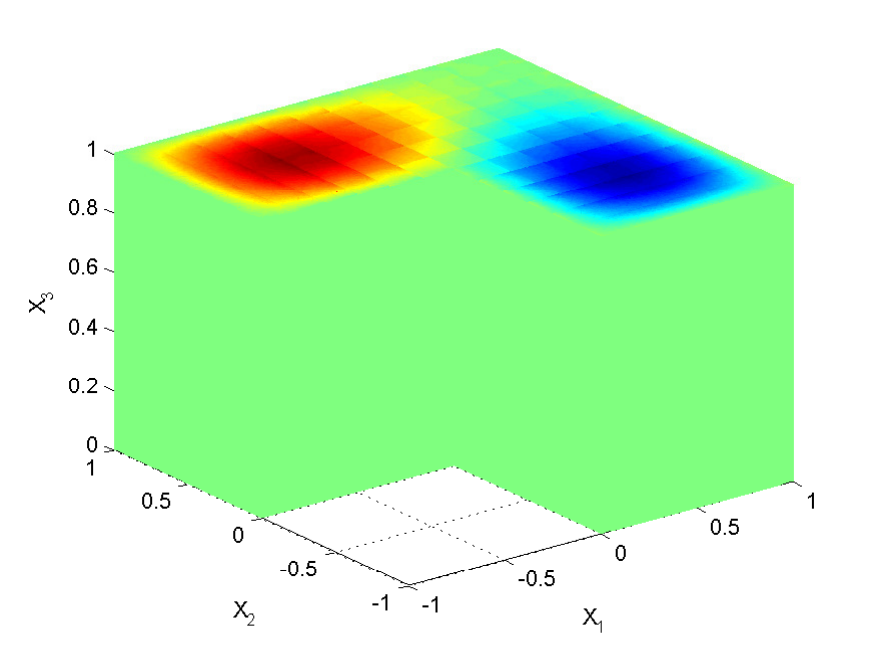}\\
\includegraphics[width=4.2cm, height=4.2cm]{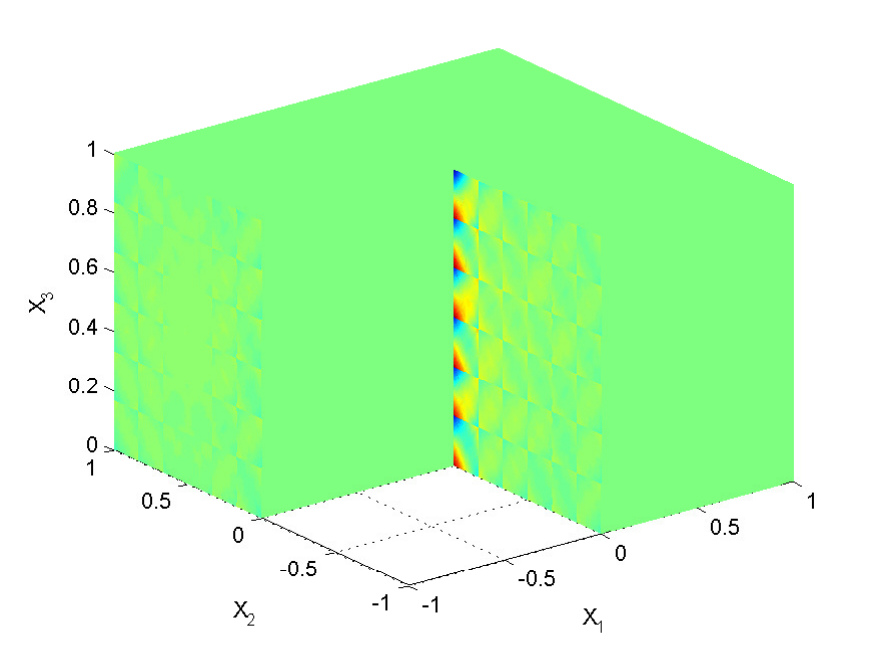}
\includegraphics[width=4.2cm, height=4.2cm]{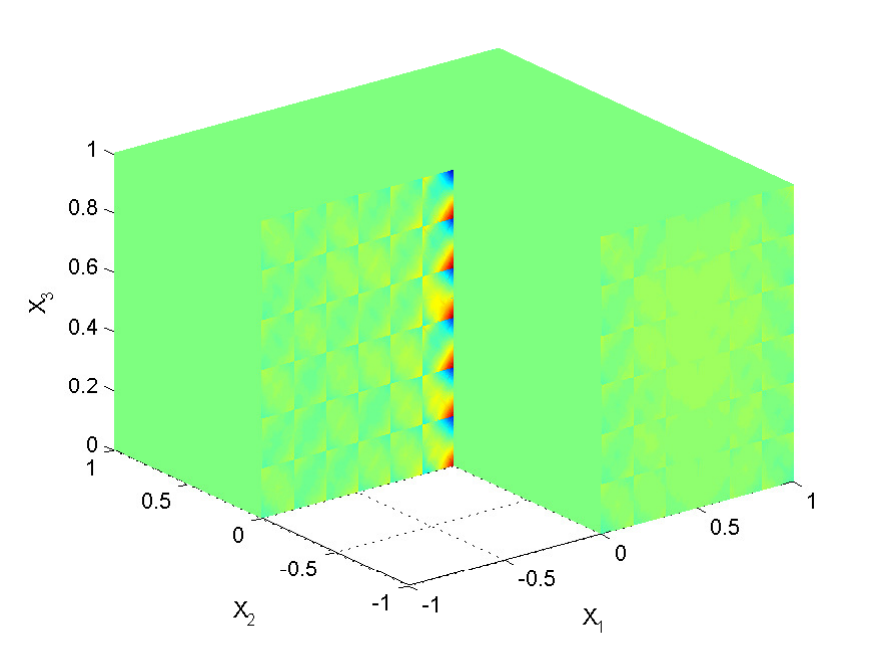}
\includegraphics[width=4.2cm, height=4.2cm]{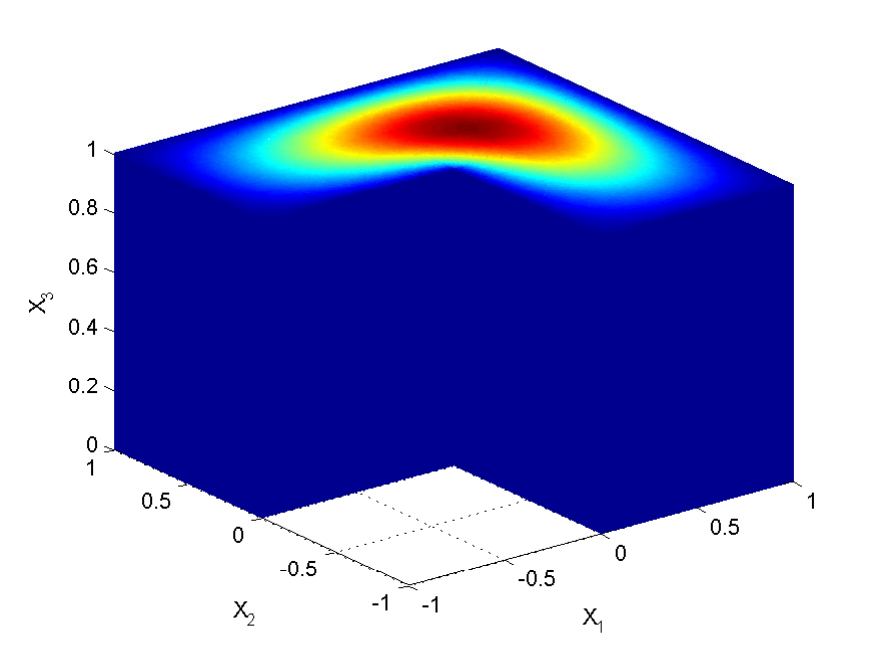}\\
\includegraphics[width=4.2cm, height=4.2cm]{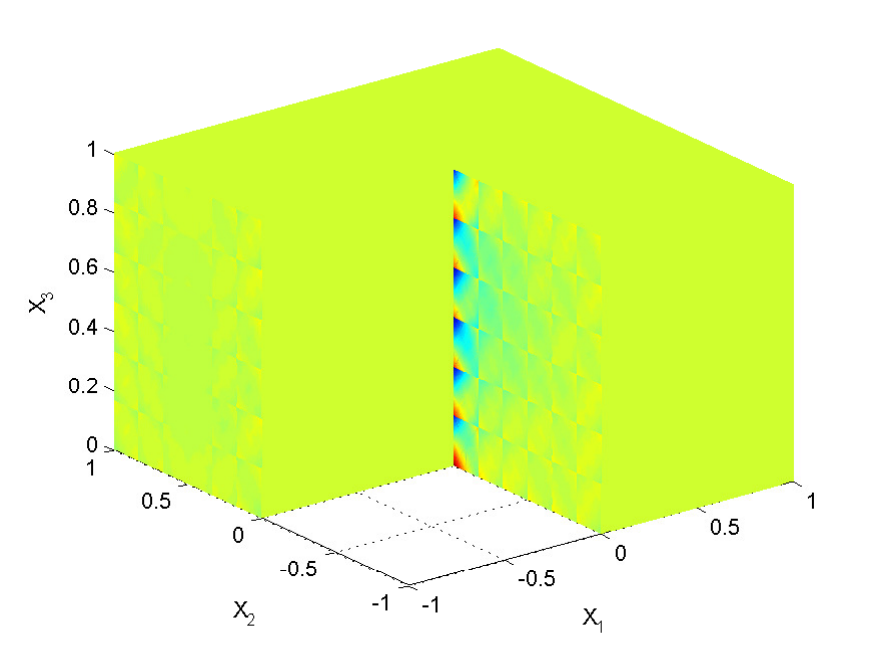}
\includegraphics[width=4.2cm, height=4.2cm]{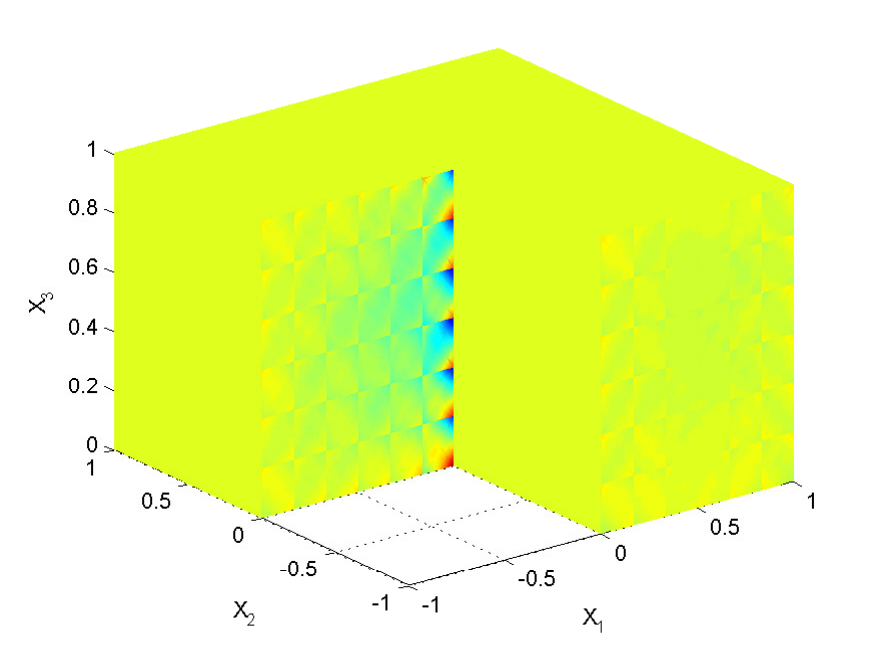}
\includegraphics[width=4.2cm, height=4.2cm]{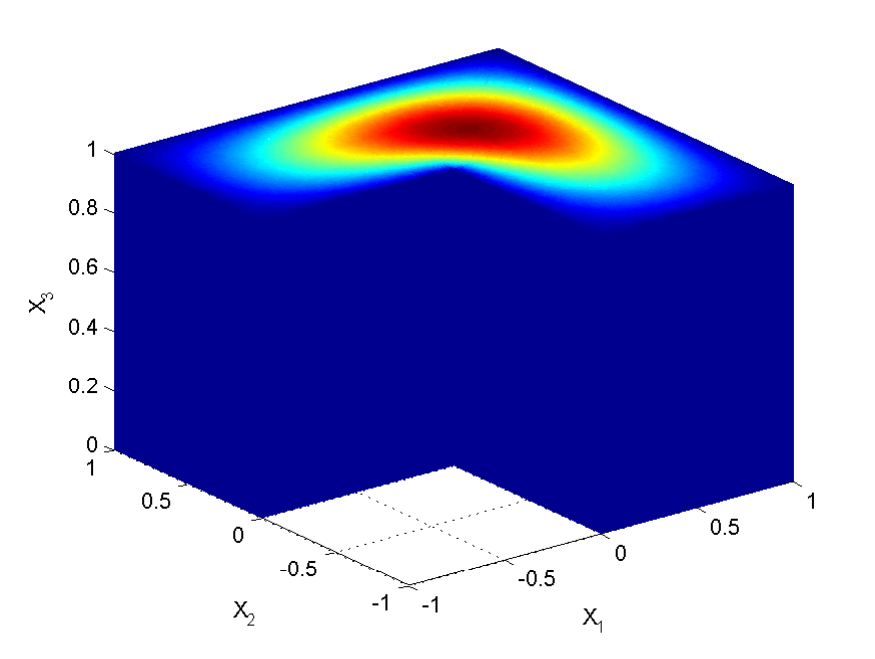}
\end{center}
\end{figure}

\end{document}